\documentclass[onefignum,onetabnum]{siamart190516}

\usepackage{diagbox}
\usepackage{times,amsmath,amsbsy,amssymb,amscd,mathrsfs}
\usepackage{cases}
\usepackage{longtable}
\usepackage{array} 
\usepackage{algorithm}
\usepackage{algorithmic}
 \usepackage{bm}
\usepackage{longtable}
\usepackage{pgfplots}
\usepackage{diagbox}
\usepackage{booktabs,makecell,multirow}
 \usepackage{url}
 \usepackage{epstopdf}

\crefname{equation}{}{}
\crefname{lem}{Lemma}{Lemmas}
\crefname{thm}{Theorem}{Theorems}
\crefname{assum}{Assumption}{Assumptions}
\crefname{rem}{Remark}{Remarks}
\crefname{line}{Line}{Lines}

\newcommand{\nm}[1]{\left\lVert {#1} \right\rVert}
\newcommand{\snm}[1]{\left\lvert {#1} \right\rvert}
\newcommand{\dual}[1]{\left\langle {#1} \right\rangle}
\newcommand{\argmin}[0]{ {\mathop{{\rm  argmin}}\,}}

\newcounter{mnote}
\let\oldmarginpar\marginpar
\renewcommand\marginpar[1]{\-\oldmarginpar[\raggedleft\footnotesize #1]%
	{\raggedright\footnotesize #1}}

\newsiamremark{assum}{Assumption} 
\newsiamthm{lem}{Lemma} 
\newsiamthm{prop}{Proposition} 
\newsiamthm{thm}{Theorem} 
\newsiamthm{coro}{Corollary} 
\newsiamthm{Def}{Definition} 
\newsiamthm{remark}{Remark}
\newsiamthm{rem}{Remark}

\newcommand{\dd}{\,{\rm d}}

\newcommand{\R}{\,{\mathbb R}}

\newcolumntype{I}{!{\vrule width 1,5pt}}
\newlength\savedwidth

\newlength\savewidth

\newcommand{\uapd}[0]{ \mathtt{sub}\textrm{-}\mathtt{UAPD}}
\newcommand{\whk}[0]{\widehat{\lambda}_k}

\usepackage{graphicx}
\usepackage{tikz}
\tikzset{elegant/.style={smooth,thick,samples=50,black}}
\tikzset{eaxis/.style={->,>=stealth}}

\begin{document}

\title{A universal accelerated primal-dual method for convex optimization problems\thanks{This work was supported by the Doctoral Starting up Foundation of  Chongqing Normal University (No. 202210000161).}}
	

\author{Hao Luo\thanks{National Center for Applied Mathematics in Chongqing, Chongqing Normal University, Chongqing, 401331, China (\email{luohao@cqnu.edu.cn})}
}

\headers{A universal accelerated primal-dual method}{A universal accelerated primal-dual method}


\maketitle

	\begin{abstract}
	This work presents a universal accelerated first-order primal-dual method for affinely constrained convex optimization problems. It can handle both Lipschitz and H\"{o}lder gradients but does not need to know the smoothness level of the objective function. In line search part, it uses dynamically decreasing parameters and produces approximate Lipschitz constant with moderate magnitude. In addition, based on a suitable discrete Lyapunov function and tight decay estimates of some differential/difference inequalities, a universal optimal mixed-type convergence rate is  established. Some numerical tests are provided to confirm the efficiency of the proposed method.
\end{abstract}

	

%
%

	\tableofcontents

\section{Introduction}
Consider the minimization problem 
\begin{equation}\label{eq:min-h-g-A}
	\min_{x\in Q}\, \left\{f(x) := h(x) + g(x):\,Ax = b\right\},
\end{equation}
where $(A,b)\in\R^{m\times  n}\times\R^m$, $Q\subset \R^n$ is a simple closed convex subset, and $f:\R^n\to\R\cup\{+\infty\}$ is properly closed and convex, with smooth part $h $ and nonsmooth part $g$. The model problem \cref{eq:min-h-g-A} arises from many practical applications, such as compressed sensing \cite{candes_comp_2006}, image processing \cite{chambolle_introduction_2016} and decentralized distributed optimization \cite{boyd_distributed_2010}. 

In the literature, existing algorithms mainly include Bregman iteration \cite{cai_linearized_2009,huang_accelerated_2013,yin_bregman_2008}, quadratic penalty method \cite{Lan2013,li_convergence_2017}, augmented Lagrangian method (ALM)  \cite{he_acceleration_2010,HE2022110547,he_inertial_2022,kang_inexact_2015,luo_primal-dual_2022,tao_accelerated_2016,tran-dinh_constrained_2014,tran-dinh_primal-dual_2015}, and alternating direction method of multipliers  \cite{goldfarb_fast_2013,goldstein_fast_2014,kadkhodaie_accelerated_2015,Li2019,luo_unified_2021,ouyang_accelerated_2015,sabach_faster_2022,tian_alternating_2018,tran-dinh_proximal_2019,tran-dinh_augmented_2018,tran-dinh_non-stationary_2020,Xu2017}. Generally speaking, these methods have sublinear rate $\mathcal O(1/k)$ for convex problems and can be further accelerated to $\mathcal O(1/k^2)$ for (partially) strongly convex objectives. We also note that primal-dual methods 
\cite{chambolle_first-order_2011,esser_general_2010,he_convergence_2014,jiang_approximate_2021,tran-dinh_unified_2021,giselsson_smoothing_2018,tran-dinh_smooth_2018,valkonen_inertial_2020} and operator splitting algorithms  \cite{davis_convergence_2016,eckstein_splitting_1989,eckstein_douglas-rachford_1992,bauschke_2019} can be applied to \cref{eq:min-h-g-A} with two-block structure.

However, among these works, it is rare to see the optimal {\it mixed-type} convergence rate, i.e., the lower complexity bound \cite{ouyang_lower_2021}
\begin{equation}\label{eq:opt-bd}
	\min\left\{ \frac{\nm{A}}{\epsilon},\,\frac{\nm{A}}{\sqrt{\mu\epsilon}}\right\}+
	\min\left\{\sqrt{L/\epsilon},\,\sqrt{L/\mu}\cdot\snm{\ln\epsilon}\right\},
\end{equation}
where $\mu\geq 0$ is the convexity parameter of $f$ and $L$ is the Lipschitz constant of $\nabla h$. Both Nesterov's smoothing technique \cite{nesterov_smooth_2005} and the accelerated primal-dual method in \cite{chen_optimal_2014} achieve the lower bound for convex case $\mu=0$. The inexact ALM framework in \cite{xu_iteration_2021} possesses the optimal complexity \cref{eq:opt-bd} but involves a subroutine for inexactly solving the subproblem.

We mention that the second part of \cref{eq:opt-bd} corresponds to the objective $f$ and agrees with the well-known lower complexity bound of first-order methods for solving unconstrained convex problems with Lipschitz gradients. The intermediate non-Lipschitz case is also of interest to be considered \cite{nbmirovskii_optimal_1985,nesterov_introductory_2004}. Particularly, when $\nabla f$ is H\"{o}lder continuous (cf.\cref{eq:holder-h}) with exponent $\nu\in[0,1)$, Nesterov \cite{nesterov_universal_2015} presented a universal fast gradient method (FGM) that did {\it not} require \`{a} priori knowledge of the smoothness parameter $\nu$ and the H\"{o}lderian constant $M_\nu(f)$. A key ingredient of FGM is that H\"{o}lderian gradients can be recast into the standard Lipschitz case but with inexact computations \cite{devolder_first-order_2014,stonyakin_gradient_2019,stonyakin_generalized_2022,stonyakin_inexact_2020}, and it achieves the 
optimal complexity \cite{nemirovsky_problem_1983}
\begin{equation}\label{eq:bd-f}
	\left(\frac{M_\nu(f)}{\epsilon}\right)^\frac{2}{1+3\nu}.
\end{equation}
More extensions of FGM can be found in \cite{guminov_universal_2019,guminov_primal-dual_2018,kamzolov_universal_2019}.

The dual problem of \cref{eq:min-h-g-A} reads equivalently as
\begin{equation}\label{eq:dual-phi}
	\min_{\lambda\in \R^{m}}\,\left\{\varphi(\lambda): = \dual{b,\lambda}+\max_{x\in Q}\left\{ \dual{-A^\top \lambda,x}-f(x)\right\}\right\}.
\end{equation}
If $f$ is uniformly convex of degree $p\geq 2$ (see \cite[Definition 1]{nesterov_universal_2015}), then $\nabla \varphi$ is H\"{o}lder continuous with exponent $\nu = 1/(p-1)$ (cf. \cite[Lemma 1]{nesterov_universal_2015}). The methods in \cite{dvurechensky_computational_2018,lin_efficient_2019} work for strongly convex problems, i.e., the Lipschitzian case ($\nu = 1$). Yurtsever et al. \cite{yurtsever_universal_2015} proposed an accelerated universal primal-dual gradient method (AccUniPDGrad) for general H\"{o}lderian case ($\nu<1$) and established the complexity bound \cref{eq:bd-f} for objective residual and feasibility violation, with $M_\nu(f)$ being replaced with $M_\nu(\varphi)$. Similarly with the spirit of FGM, the proposed method utilizes the ``inexactness'' property of $\nabla \varphi$ and applies FISTA \cite{beck_fast_2009} to \cref{eq:dual-phi} with a backtracking line search procedure.

In this work, we propose a universal accelerated primal-dual method (see \cref{algo:UAPD}) for solving \cref{eq:min-h-g-A}. Compared with existing works, the main contributions are highlighted as follows:
\begin{itemize}
	\item It is first-order black-box type for both Lipschitz and H\"{o}lder cases but does not need to know the smoothness level priorly. 
	\item It is equipped with the Bregman divergence and can handle the non-Euclidean setting.
	\item In line search part, it adopts dynamically decreasing tolerance while FGM \cite{nesterov_universal_2015} and AccUniPDGrad \cite{yurtsever_universal_2015} use the desired fixed accuracy.
	\item By using the tool of Lyapunov function and tight decay estimates of some differential/difference inequalities, we prove the universal mixed-type estimate that achieves the optimal complexity (including \cref{eq:opt-bd} as a special case).
\end{itemize}
We also provide some numerical tests to validate the practical performance. It is confirmed that: (i) the proper choice of Bregman distance is crucial indeed; (ii) our method outperforms FGM and AccUniPDGrad especially for non-Lipschitz problems and smooth problems with large Lipschitz constants, as the automatically decreasing tolerance leads to approximate Lipschitz constants with moderate magnitude. 

Our method here is motivated from an implicit-explicit time discretization of a novel accelerated Bregman primal-dual dynamics (see \cref{eq:apd-sys-app}), which is an extension of the previous accelerated primal-dual flow \cite{luo_acc_primal-dual_2021} to the non-Euclidean case. For unconstrained problems, there are some existing continuous dynamics \cite{krichene_accelerated_2015,wibisono_variational_2016,Wilson_2018} with Bregman divergence. For linearly constrained case, we see an accelerated primal-dual mirror model \cite{zhao_accelerated_2022}, which is inspired by the accelerated mirror descent \cite{krichene_accelerated_2015} and primal-dual dynamical approach \cite{feijer_stability_2010} but without numerical discretizations.

The rest of the paper is organized as follows. In \cref{sec:pre} we provide some preliminaries including Bregman divergence and H\"{o}lder continuity. Then the main algorithm together with its universal mixed-type estimate is presented in \cref{sec:pf-thm1}, and rigorous proofs of two technical lemmas are summarized in \cref{sec:pf-lem-diff-Lk,sec:pf-lem-est-tk}, respectively. Finally, some numerical results are reported in \cref{sec:num}.

\section{Preliminary}
\label{sec:pre}
\subsection{Notations}
Let $\dual{\cdot,\cdot}$ be the usual inner product of vectors and $\nm{\cdot}$ be the standard Euclidean norm (of vectors and matrices). 
Given a proper function $g:\R^n\to \R \cup\{+\infty\}$, the subdifferential of $g$ at any $x\in \R^n$ is the set of all subgradients:
\[
\partial g(x): = \left\{\xi\in \R^n:\,g(y)\geq g(x)+\dual{\xi,y-x}\quad\forall\,y\in \R^n\right\}.
\]
Recall that $Q\subset \R^n$ is a nonempty closed convex subset. We denote by  $\iota_Q(\cdot)$ the indicator function of $Q$ and let $N_Q(\cdot): =\partial \iota_Q(\cdot)$ be its normal cone.

Introduce the Lagrangian for the model problem \cref{eq:min-h-g-A}:
\[
\mathcal L(x,\lambda): = f(x) +\iota_Q(x)+ \dual{\lambda,Ax-b}\quad \forall\,(x,\lambda)\in\R^n\times \R^m.
\]
We say $(x^*,\lambda^*)\in Q\times \R^m$ is a saddle point of $\mathcal L$ if
\[
\mathcal L(x^*,\lambda)\leq \mathcal L(x^*,\lambda^*)\leq \mathcal L(x,\lambda^*)\quad \forall\,(x,\lambda)\in\R^n\times \R^m,
\]
which also implies the optimality condition:
\[
Ax^*-b=0,\quad \partial f(x^*) + N_Q(x^*)+A^\top \lambda^*\ni0.
\]
\subsection{Bregman divergence}
\label{sec:Breg}
Let $\phi:Q\to\R$ be a smooth {\it prox-function} and consider the corresponding
{\it Bregman divergence}
\begin{equation*}
	D_\phi(x,y) := \phi(x)-\phi(y)-\dual{\nabla \phi(y),x-y}\quad\forall\,x,\,y\in Q.
\end{equation*}
Suppose $\phi$ is $1$-strongly convex, which means 
\begin{equation}\label{eq:D-phi}
	D_\phi(x,y)\geq \frac{1}{2}\nm{x-y}^2\quad\forall\,x,\,y\in Q.
\end{equation}
Particularly, $\phi(x) = 1/2\nm{x}^2$ leads to $D_\phi(x,y)=D_\phi(y,x)=1/2\nm{x-y}^2$,
which boils down to the standard Euclidean setting.
In addition, 
we have the following {\it three-term identity};  see \cite[Lemma 3.2]{Chen-conv-1993} or \cite[Lemma 3.3]{Dvurechensky2021}.
\begin{lem}[\cite{Chen-conv-1993,Dvurechensky2021}]
	\label{lem:3-term-id}
	For any $x,y,z\in Q$, it holds that
	\begin{equation}\label{eq:3-term-id}
		\dual{\nabla \phi(x)-\nabla \phi(y),y-z}
		=	D_\phi(z,x) - 	D_\phi(z,y)- 	D_\phi(y,x).
	\end{equation}
	If $\phi(x) = \frac{1}{2}\nm{x}^2$, then 
	\begin{equation}\label{eq:3-term-id-quad}
		2\dual{x-y,y-z}
		=	\nm{x-z}^2 - 	\nm{y-z}^2- \nm{x-y}^2.
	\end{equation}
\end{lem}
\subsection{H\"{o}lder continuity}
Let $h$ be any differentiable function on $Q$. For $0\leq \nu\leq 1$, define 
\[
M_\nu(h) : = \sup_{\substack{x\neq y\\x,\,y\in Q}}\frac{\nm{\nabla h(x)-\nabla h(y)}}{\nm{x-y}^\nu}.
\]
If $M_\nu(h)<\infty$, then $\nabla h$ is H\"{o}lder continuous with exponent $\nu$:
\begin{equation}\label{eq:holder-h}
	\nm{\nabla h(x)-\nabla h(y)}\leq M_\nu(h)\nm{x-y}^\nu
	\quad\forall\,x,y\in Q,
\end{equation}
and this also implies 
\begin{equation}\label{eq:holder}
	h(x)\leq h(y)+\dual{\nabla h(y),x-y}+\frac{M_\nu(h)}{1+\nu}\nm{x-y}^{1+\nu}
	\quad\forall\,x,y\in Q.
\end{equation}
For $\nu = 1$, $M_1(h)$ corresponds to the Lipschitz constant of $\nabla h$, and we also use the conventional notation $L_h = M_1(h)$.

According to \cite[Lemma 2]{nesterov_universal_2015}, the estimate \cref{eq:holder} can be transferred into the usual gradient descent inequality, with ``inexact computations''. Based on this, (accelerated) gradient methods can be used to minimize functions with H\"{o}lder continuous gradients \cite{devolder_first-order_2014,stonyakin_gradient_2019,stonyakin_generalized_2022,stonyakin_inexact_2020}. 
\begin{prop}[\cite{nesterov_universal_2015}]
	\label{prop:M-delta}	
	Assume $M_\nu(h)<\infty$ and define
	\begin{equation}\label{eq:M-nu-delta}
		M(\nu,\delta): =\delta^{\frac{\nu-1}{\nu+1}}[M_\nu(h)]^{\frac{2}{\nu+1}}
		\quad\forall\,\delta>0.
	\end{equation}
	Then for any $M\geq M(\nu,\delta)$, we have
	\begin{equation*}
		h(x)\leq h(y)+\dual{\nabla h(y),x-y}
		+\frac{M}{2}\nm{x-y}^2+\frac{\delta}{2}
		\quad\forall\,x,y\in Q.
	\end{equation*}
\end{prop}

\section{Main Algorithm}
\label{sec:pf-thm1}
Throughout, we make the following assumption on $f = h+g$:
\begin{assum}\label{assum:h-g}
	The nonsmooth part	$g$ is properly closed and convex on $Q$. 
	The smooth part $h$ satisfies $	\inf_{0\leq \nu\leq 1}M_\nu(h)<\infty$ and is $\mu$-convex on $Q$ with $\mu\geq0$, i.e.,
	\[
	h(x)\geq h(y)+\dual{\nabla h(y),x-y}+\mu D_\phi(x,y)
	\quad\forall\,x,y\in Q.
	\]
\end{assum}

\begin{algorithm}[H]
	\caption{Universal Accelerated Primal-Dual (UAPD) Method}
	\label{algo:UAPD}
	\begin{algorithmic}[1] 
		\REQUIRE  $\beta_0 =1,\,\gamma_0,\,M_{0}>0,\,\mu\geq 0$ and $\nm{A}$.
		\STATE Initialization: $x_0,\,v_0\in Q$ and $\lambda_0\in \R^m$.
		\FOR{$k=0,1,\cdots$}
		\STATE Set $i_k = 0,\,M_{k,0} = M_{k}$ and $S_k=\{x_k,v_k,\lambda_k,\beta_k,\gamma_k\}$.
		\STATE $(y_{k,i_k},x_{k,i_k},v_{k,i_k},\alpha_{k,i_k},\delta_{k,i_k},\Delta_{k,i_k}) =\uapd(k,S_k,M_{k,i_k})$.\label{algo:sub-uapd}
		\WHILE[Line search]{$h(x_{k,i_k})-\Delta_{k,i_k}> \delta_{k,i_k}/2$}		
		\label{algo:line-search}		
		\STATE Set $i_k = i_k + 1$ and $M_{k,i_k} = 2^{i_k}M_{k,0}$.		
		\STATE $(y_{k,i_k},x_{k,i_k},v_{k,i_k},\alpha_{k,i_k},\delta_{k,i_k},\Delta_{k,i_k})
		=\uapd(k,S_k,M_{k,i_k})$.
		\label{algo:sub-uapd-2}
		\ENDWHILE
		\STATE Set $\alpha_k =  \alpha_{k,i_k},\,M_{k+1} = M_{k,i_k}$ and $\delta_{k+1} = \delta_{k,i_{k}}$.
		\STATE Update $\gamma_{k+1} = (\gamma_k+\mu\alpha_k)/(1+\alpha_k)$ and $\beta_{k+1} = \beta_k/(1+\alpha_k)$.
		\STATE Update $x_{k+1} =x_{k,i_k},v_{k+1} =v_{k,i_k}$ and $\lambda_{k+1}=
		\lambda_k+\alpha_k/\beta_k\left(
		Av_{k+1}-b
		\right)$.
		\ENDFOR
	\end{algorithmic}
\end{algorithm}
\begin{algorithm}[H]
	\caption{$(\widetilde{y}_k,\widetilde{x}_k,\widetilde{v}_k,\widetilde{\alpha}_k,\widetilde{\delta}_k,\widetilde{\Delta}_k) =\uapd(k,S_k,\widetilde{M}_{k})$}
	\label{algo:UAPD-inner}
	\begin{algorithmic}[1] 
		\REQUIRE $k\in\mathbb N,\,\widetilde{M}_k>0$ and $S_{k} = \{x_k,v_k,\lambda_k,\beta_k,\gamma_k\}$.
		\STATE Choose the step size $\widetilde{\alpha}_k = \sqrt{\beta_k\gamma_k}/\sqrt{\beta_k\widetilde{M}_k
			+\nm{A}^2}$.
		\STATE Set $\widetilde{\beta}_k =\beta_k/(1+\widetilde{\alpha}_k)$ and $\widetilde{\delta}_k =  \widetilde{\beta}_k/(k+1)$.
		\STATE Set $\widetilde{y}_k={}(x_k+\widetilde{\alpha}_kv_k)/(1+\widetilde{\alpha}_k)$ and $\widetilde{\lambda}_k
		=\lambda_k+\widetilde{\alpha}_k/\beta_k\left(
		Av_k-b		\right)$.
		\STATE Update $\displaystyle  \widetilde{x}_k={}(x_k+\widetilde{\alpha}_k\widetilde{v}_k)/(1+\widetilde{\alpha}_k)$ with
		\[
		\widetilde{v}_k = \mathop{\argmin}\limits_{v\in Q}\left\{
		g(v)+
		\big\langle  \nabla h(\widetilde{y}_k)+A^\top\widetilde{\lambda}_k,v\big\rangle +\mu D_\phi(v,\widetilde{y}_k)+ \frac{\gamma_k}{\widetilde{\alpha}_k}D_\phi(v,v_k)
		\right\}.
		\]
		\STATE Compute $		\widetilde{\Delta}_k = h(\widetilde{y}_k)+\dual{\nabla h(\widetilde{y}_k),\widetilde{x}_k-\widetilde{y}_k} + \frac{\widetilde{M}_k}{2}\nm{\widetilde{x}_k-\widetilde{y}_k}^2$.
	\end{algorithmic}
\end{algorithm}

Our main algorithm, called Universal Accelerated Primal-Dual (UAPD) method, is summarized in \cref{algo:UAPD}, where the subpart $\uapd$ in lines \ref{algo:sub-uapd} and \ref{algo:sub-uapd-2} has been given by \cref{algo:UAPD-inner}. Note that we do {\it not} require priorly the smoothness constant $M_\nu(h)$ but perform a line search procedure. 
\subsection{Line search}
From line \ref{algo:line-search} of \cref{algo:UAPD}, we find that $i_k$ is the smallest integer such that 
\[\small
\begin{aligned}
	h(x_{k,i_k})\leq h(y_{k,i_k})+\dual{\nabla h(y_{k,i_k}),x_{k,i_k}-y_{k,i_k}} + \frac{M_{k,i_k}}{2}\nm{x_{k,i_k}-y_{k,i_k}}^2+\frac{\delta_{k,i_{k}}}{2}.
\end{aligned}
\]
We claim that $i_k$ is finite for each $k\in\mathbb N$. Indeed, $M_{k,i_k} = 2^{i_k}M_{k,0}$ increases as $i_k$ does, and the step size 
\begin{equation}\label{eq:ak-ik}
	\alpha_{k,i_k} = \frac{\sqrt{\beta_k\gamma_{k}}}{\sqrt{\beta_k M_{k,i_k} + \nm{A}^2}}
\end{equation}
has to be decreasing. Thus the tolerance
\begin{equation}\label{eq:deltak-ik}
	\delta_{k,i_k} = \frac{1}{k+1}\cdot\frac{\beta_{k}}{1+\alpha_{k,i_k}}
\end{equation}
is increasing and by \cref{eq:M-nu-delta}, $M(\nu,\delta_{k,i_k})$ is decreasing . This together with \cref{prop:M-delta,assum:h-g} concludes that either $i_k = 0$ or $1\leq i_k\leq s^*+1$ where $s^*\geq 0$ solves $M_{k,s^*} = M(\nu,\delta_{k,s^*})$; see \cref{fig:Mik}. Moreover, we notice that 
\[
M_{k,s^*}= 2^{s^*}M_{k,0}\leq M(\nu,\delta_{k,0})\quad\Longrightarrow\quad 
s^*\leq \log_2\frac{M(\nu,\delta_{k,0})}{M_{k,0}}<\infty.
\]

\begin{figure}[H]
	\centering
	\begin{tikzpicture}
		\draw[eaxis] (-1.8,0) -- (7,0) node[below] {$s$};
		\draw[eaxis] (0,-1) -- (0,5.2) node[left] {};
		
		\path (-0.2,0) node(axis0)[below]{$O$};	

		\path (3,3.6)node(axis0)[left]{\small $M_{k,s}$};	
		
		\path (4.5,1.6)node(axis0)[left]{\small $M(\nu,\delta_{k,s})$};	
		
		\path (0,2)node(axis0)[left]{\small $M_{k,0}$};	
		\path (0,1+2.5)node(axis0)[left]{\small $M(\nu,\delta_{k,0})$};	
		\path (0,1)node(axis0)[left]{\small$M(\nu,\delta_{k,\infty})$};	
		\draw[dashed,black] (0,1) -- (6,1); 
		
		\path (0.9,0)   node(axis0)[below]{$s^*$};	
		\draw[dashed,black] (0.9,0) -- (0.9,2.35); 
		
		\fill (0.9,2.35) circle (1.3pt);
		
		\draw[elegant,domain=0:4] plot(\x,{2*1.2^\x});
		\draw[elegant,domain=0:6] plot(\x,{1+2.5/2^\x});	
	\end{tikzpicture}
	
	\caption{Illustrations of $M_{k,s}$ and $M(\nu,\delta_{k,s})$ as functions of $s\in[0,\infty)$. Here $\delta_{k,\infty} =\lim\limits_{s\to\infty}\delta_{k,s}= \beta_k/(k+1)$ since $\alpha_{k,s}\to0$ as $s\to\infty$.}
	\label{fig:Mik}
\end{figure}
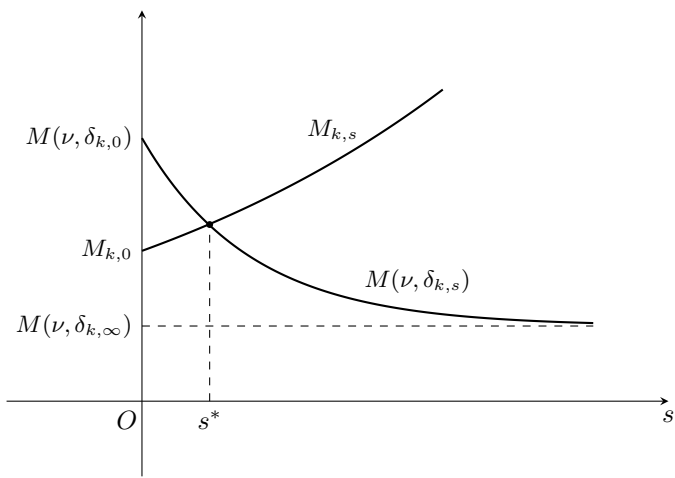
\begin{rem}
	\label{rem:compare-line-search}	
	In the line search part, \cref{algo:UAPD} adopts dynamically decreasing tolerance \cref{eq:deltak-ik}, i.e., $\delta_k = \beta_{k}/k$. However, the methods in  \cite{nesterov_universal_2015} and \cite[Algorithm 2]{yurtsever_universal_2015} chose $\delta_k=\epsilon/k$, where $\epsilon$ is the desired accuracy. Hence, by \cref{prop:M-delta}, the approximate smoothness constant $M_k$ of our algorithm is smaller than these two methods, especially for H\"{o}lderian case. This will be verified by numerical experiments.
\end{rem}

Below, we give an upper bound of $M_k$ and the total number of line search steps. By \cref{thm:conv}, $\beta_k$ corresponds to the convergence rate of \cref{algo:UAPD} and admits explicit decay estimate with respect to $k$ (see \cref{lem:est-tk}). If the desired accuracy $\beta_{k+1}=\mathcal O(\epsilon)$ is given, then the term $\snm{\log_2\beta_{k+1}}$ in \cref{eq:lg2-Mnu} can also be replaced by $\snm{\log_2\epsilon}$.
\begin{lem}
	\label{lem:Mk}
	For any $k\in\mathbb N$, we have
	\begin{equation}\label{eq:Mk}
		M_{k+1}\leq \max\left\{ 2\sqrt{2}M(\nu,\delta_{k+1}),\,M_{0}\right\},
	\end{equation}
	and consequently, it holds that 
	\begin{equation}\label{eq:bd-ik}
		\sum_{j=0}^{k}i_j\leq k+1+\max\left\{1,  \log_2\frac{M(\nu,\delta_{k+1})}{M_{0}/2\sqrt{2}}\right\},
	\end{equation}
	where 
	\begin{equation}\label{eq:lg2-Mnu}\small
		\log_2\frac{M(\nu,\delta_{k+1})}{M_{0}/2\sqrt{2}}= 
		\log_2\frac{[M_\nu(h)]^{\frac{2}{1+\nu}}}{M_0/2\sqrt{2}}+
		\frac{1-\nu}{1+\nu}\left(	 \log_2(k+1)+\snm{\log_2\beta_{k+1}}\right).
	\end{equation}
\end{lem}
\begin{proof}
	See \cref{sec:app-Mk}.
\end{proof}
\subsection{Time discretization interpretation}
Below, we provide a time discretization interpretation of \cref{algo:UAPD}. Given the $k$-th iterations $(x_k,v_k,\lambda_k)$ and the parameters $(\gamma_k,\beta_k,M_{k})$, the line search procedure produces $(y_k,x_{k+1},v_{k+1})$ that satisfy
\begin{subnumcases}{\label{eq:pc-imex}}
	\label{eq:pc-imex-y}
	\frac{y_k-x_k}{\alpha_k} ={} v_k-y_k,\\
	\label{eq:pc-imex-v}
	\gamma_{k}\frac{\nabla\phi(v_{k+1})-\nabla\phi(v_k)}{\alpha_k}\in{}
	\mu\left[\nabla\phi(y_k)-\nabla\phi(v_{k+1})\right]-\mathcal G(y_k,v_{k+1},\lambda_k),\\
	\label{eq:pc-imex-x}
	\frac{x_{k+1}-x_k}{\alpha_k} ={} v_{k+1}-x_{k+1},\\
	\label{eq:pc-imex-l}
	\beta_{k}\frac{\lambda_{k+1}-\lambda_k}{\alpha_k}={}Av_{k+1}-b,
\end{subnumcases}
where $\mathcal G(y_k,v_{k+1},\lambda_k):=\nabla h(y_k)+\partial g(v_{k+1})+N_Q(v_{k+1})+A^\top \whk$
with $\whk = \lambda_k+\alpha_k/\beta_k(Av_k-b)$, and the step size  $\alpha_k$ solves (cf.\cref{eq:ak-ik})
\begin{equation}\label{eq:ak}
	\alpha_k^2(\beta_{k}M_{k+1}
	+\nm{A}^2)=  \gamma_k\beta_k.
\end{equation} 
Besides, $y_k$ and $x_{k+1}$ satisfy
\begin{equation}\label{eq:xk1-yk}
	h(x_{k+1})\leq  h(y_k)+\dual{\nabla h(y_k),x_{k+1}-y_k}+\frac{M_{k+1}}{2}\nm{x_{k+1}-y_k}^2+\frac{\delta_{k+1}}{2},
\end{equation}
and the parameters $(\gamma_{k+1},\beta_{k+1})$ are governed by 
\begin{equation}\label{eq:gama-betak}
	\frac{\gamma_{k+1}-\gamma_k}{\alpha_k} = {}\mu-\gamma_{k+1},\quad 	\frac{\beta_{k+1}-\beta_k}{\alpha_k} = -\beta_{k+1},
\end{equation}
with $\beta_0 = 1$ and $\gamma_0>0$.

As one can see, $y_k$ in \eqref{eq:pc-imex-y} is an intermediate which provides a ``prediction'', and then the ``correction'' step \eqref{eq:pc-imex-x} is used to update $x_{k+1}$.
From \eqref{eq:pc-imex-y}, \eqref{eq:pc-imex-v}, and \eqref{eq:pc-imex-x}, it is not hard to find that $y_k,\,v_{k+1},\,x_{k+1}\in Q$, as long as $x_k,v_k\in Q$. Therefore, with $x_0,v_0\in Q$, it holds that $\{x_k,y_k,v_k\}_{k\in\mathbb N}\subset Q$. 

Furthermore, 
we mention that the reformulation \cref{eq:pc-imex} admits an implicit-explicit time discretization for the following primal-dual dynamics:
\begin{equation}\label{eq:apd-sys-app}
	\left\{
	\begin{aligned}
		{}&		x' = v-x,\\
		{}&		\gamma \frac{\dd}{\dd t} \nabla \phi(v){}\in \mu (\nabla\phi(x)-\nabla\phi(v))-\big(\partial f(x)+N_Q(x)+A^\top \lambda\big),\\
		{}&		\beta \lambda' {}= Av-b,
	\end{aligned}
	\right.
\end{equation}
where $\gamma$ and $\beta$ are governed by continuous  analogues to \cref{eq:gama-betak}:
\begin{equation}\label{eq:gama-beta}
	\gamma' = \mu-\gamma,\quad \beta' =-\beta.
\end{equation}
We call \cref{eq:apd-sys-app} the {\it Accelerated Bregman Primal-Dual} (ABPD) flow.
In the standard Euclidean setting $\phi(x) = 1/2\nm{x}^2$, it amounts to the accelerated primal-dual flow proposed in \cite{luo_acc_primal-dual_2021}. For well-posedness and exponential decay estimate of \cref{eq:apd-sys-app} with smooth objective $f$ and general prox-function $\phi$, we refer to \cref{app:abpd}. 
\subsection{A universal estimate}
Let $\{(x_k,v_k,\lambda_k,\gamma_k,\beta_k)\}_{k\in\mathbb N}$ be the sequence generated from \cref{algo:UAPD}. We introduce the discrete Lyapunov function 
\begin{equation}\label{eq:Ek}
	\mathcal E_k: =  \mathcal L\left(x_{k},\lambda^*\right)
	-\mathcal L\left(x^*,\lambda_k\right)+ \gamma_kD_\phi(x^*,v_k)+\frac{\beta_k}{2}\nm{\lambda_k-\lambda^*}^2.
\end{equation}
A one-step estimate is presented below.
\begin{lem}
	\label{lem:diff-Lk}
	Under \cref{assum:h-g}, we have 
	\begin{equation}\label{eq:diff-Lk}
		\mathcal E_{k+1}-		\mathcal E_{k}\leq -\alpha_k		\mathcal E_{k+1}+\frac{\delta_{k+1}}{2}(1+\alpha_k)\quad\forall\,k\in\mathbb N.
	\end{equation}
\end{lem}
\begin{proof}
	See \cref{sec:pf-lem-diff-Lk}.
\end{proof}

Using this lemma, we obtain the following theorem, which says the final convergence rate is given by the sharp decay estimate of the sequence $\{\beta_k\}_{k\in\mathbb N}$; see \cref{lem:est-tk}.
\begin{thm}\label{thm:conv}
	Under \cref{assum:h-g}, we have $\{x_k,v_k\}_{k\in\mathbb N}\subset Q$ and 
	\begin{align}
		\label{eq:est-Axk-b}
		\nm{Ax_k-b}\leq{}& \beta_k\mathcal T_{0,k},\\
		\label{eq:est-fxk-fx}	
		\snm{f(x_k)-f(x^*)}\leq {}&\beta_k\mathcal W_{0,k},	\\
		\label{eq:est-Lk}	
		\mathcal L\left(x_{k},\lambda^*\right)
		-\mathcal L\left(x^*,\lambda_k\right)
		\leq {}&	\beta_k\mathcal R_{0,k},
	\end{align}
	for all $k\in \mathbb N$, where $\mathcal R_{0,k}:=\mathcal E_0+\ln (k+1),\,\mathcal T_{0,k}:=\nm{Ax_0-b}+2\sqrt{2\mathcal R_{0,k}}$ and $\mathcal W_{0,k}:=\mathcal R_{0,k}+\nm{\lambda^*}\mathcal T_{0,k}$. Moreover, if $\mu>0$, then
	\begin{equation}\label{eq:xk-vk}
		\gamma_{\min}\nm{v_{k}-x^*}^2 + \mu\nm{x_k-x^*}^2\leq 2\beta_k\mathcal R_{0,k},
	\end{equation}
	where $\gamma_{\min} := \min\{\gamma_0,\mu\}$.
\end{thm}
\begin{proof}
	From \cref{eq:gama-betak} and the contraction estimate \cref{eq:diff-Lk} follows immediately that 
	\[
	\mathcal E_{k+1}\leq \frac{1}{1+\alpha_k}\mathcal E_k+\frac{\delta_{k+1}}{2}\quad\Longrightarrow\quad 	\mathcal E_k\leq 	\beta_k\mathcal E_0+\frac{\beta_k}{2}\sum_{i=0}^{k-1}\frac{\delta_{i+1}}{\beta_{i+1}}.
	\]
	By \cref{eq:deltak-ik}, we have $\delta_{k+1} = \beta_{k+1}/(k+1)$, which further implies 
	\begin{equation}\label{eq:est-Ek}
		\mathcal E_k\leq \beta_k\mathcal E_0+\frac{\beta_k}{2}\sum_{i=0}^{k-1}\frac{1}{i+1}
		\leq 	\beta_k\left[\mathcal E_0+\ln (k+1)\right].
	\end{equation}
	This proves \cref{eq:est-Lk,eq:xk-vk}. Following the proof of \cite[Theorem 3.1]{luo_acc_primal-dual_2021}, it is not hard to establish \cref{eq:est-Axk-b,eq:est-fxk-fx}.	Hence, we conclude the proof of this theorem.
\end{proof}
\begin{rem}
	Note that the choice \cref{eq:deltak-ik} can be replaced with 
	\[
	\delta_{k,i_k} = \frac{\delta}{k+1}\cdot\frac{\beta_{k}}{1+\alpha_{k,i_k}},\quad\delta>0.
	\]
	Then the item $\mathcal R_{0,k}$ in \cref{thm:conv} becomes $	\mathcal R_{0,k}=\mathcal E_0+\delta\ln(k+1)$ and $\delta = 1/\ln(K+1)$ cancels the logarithm factor, where $K\in\mathbb N$ is the number of iterations chosen in advance.
\end{rem}

It remains to establish the decay estimate of $\{\beta_k\}_{k\in\mathbb N}$. 	From \cref{eq:ak,eq:gama-betak}, we obtain
\begin{equation}
	\label{eq:diff-bk-}
	\beta_{k+1}-\beta_{k} = -\frac{\sqrt{\gamma_{k}\beta_k}\beta_{k+1}}{\sqrt{\beta_kM_{k+1}
			+\nm{A}^2}}.
\end{equation}
A careful investigation into this difference equation gives the desired result.
\begin{lem}
	\label{lem:est-tk}
	Assume that $M_{0}\leq [M_\nu(h)]^{\frac{2}{1+\nu}}$ and $\max\{\gamma_0,\mu\}\leq \nm{A}^2$.
	If $\mu=0$, then 
	\begin{equation}\label{eq:bk-mu-0}
		\beta_k\leq 
		C_{\nu}\left(
		\frac{\nm{A}}{\sqrt{\gamma_0} k}
		+\frac{M_\nu(h)}{\gamma_{0}^{\frac{1+\nu}{2}}k^{\frac{1+3\nu}{2}}}
		\right)\quad\forall\,k\geq 1,
	\end{equation}
	and if $\mu>0$, then for all $k\geq 1$, we have
	\begin{equation}\label{eq:bk-mu>0}
		\beta_k\leq 
		C_\nu\left\{
		\begin{aligned}
			{}&\frac{\nm{A}^2}{\gamma_{\min} k^2}
			+\frac{[M_\nu(h)]^{\frac{2}{1-\nu}}}{\gamma_{\min}^{\frac{1+\nu}{1-\nu}}k^{\frac{1+3\nu}{1-\nu}}}&&\text{ if }\nu<1,\\	
			{}&	 \frac{\nm{A}^2}{\gamma_{\min} k^2}+\exp\left(-\frac{k}{8\sqrt{3}}\sqrt{\frac{\gamma_{\min}}{L_h}}\right)&&\text{ if }\nu=1,
		\end{aligned}
		\right.
	\end{equation}
	where $\gamma_{\min} = \min\{\gamma_0,\mu\}$ and $C_\nu>0$ depends only on $\nu$.
\end{lem}
\begin{proof}
	By \cref{lem:Mk}, we have
	\[
	M_k\leq \max\left\{ 2\sqrt{2}M(\nu,\delta_{k}),\,M_{0}\right\}\quad\forall\,k\geq 1.
	\]
	In view of $\delta_1 = \beta_1 = 1/(1+\alpha_0)\leq 1$, it follows immediately that 
	\[
	M(\nu,\delta_1)\geq M(\nu,1)= [M_\nu(h)]^{\frac{2}{1+\nu}}\geq M_{0}.
	\]
	Since $\delta_k = \beta_k/k$ and $\beta_k$ is decreasing,it holds that $M(\nu,\delta_{1})\leq M(\nu,\delta_{k})$ and $M_k\leq  2\sqrt{2}M(\nu,\delta_{k})$. Plugging this into \cref{eq:diff-bk-} gives
	\begin{equation}
		\label{eq:diff-bk}
		\beta_{k+1}-\beta_{k} \leq  -\frac{\sqrt{\gamma_{k}\beta_k}\beta_{k+1}}{\sqrt{2\sqrt{2}\beta_kM(\nu,\delta_{k+1})
				+\nm{A}^2}}.
	\end{equation}
	Based on this difference inequality, we obtain \cref{eq:bk-mu-0,eq:bk-mu>0}. Missing proofs are provided in \cref{sec:pf-lem-est-tk}.
\end{proof}

According to the {\it universal mixed-type estimate} established in \cref{lem:est-tk}, our \cref{algo:UAPD} achieves the optimal complexity bound for both the unconstrained case $A = O$ and affinely constrained case $A\neq O$,  with H\"{o}lderian smoothness exponent $\nu\in[0,1]$. Detailed comparisons with existing results are summarized in order.
\begin{rem}\label{rem:complexity}
	Consider first the unconstrained case: $A = O$.
	\begin{itemize}
		\item The Lipschitzian case $\nu = 1$:
		\[
		\min\left\{\sqrt{\frac{L_h}{\epsilon}},\,\sqrt{\frac{L_h}{\mu}}\cdot\snm{\ln\epsilon}\right\}.
		\]
		This is the well-known optimal complexity bound (cf.\cite{nesterov_introductory_2004,nesterov_lectures_2018}) of first-order methods for smooth convex functions with Lipschitz continuous gradients; see \cite{chen_first_2019,chen_unified_2021,luo_accelerated_2021,luo_differential_2019,nesterov_gradient_2013}.
		\item The H\"{o}lderian case $0\leq \nu < 1$:
		\begin{equation}\label{eq:Un-nu}
			\min\left\{\left(\frac{M_\nu(h)}{\epsilon}\right)^\frac{2}{1+3\nu},\quad\left( \frac{M_\nu(h)}{\mu}\right)^{\frac{2}{1+3\nu}}\cdot \left(\frac{\mu}{\epsilon}\right)^{\frac{1-\nu}{1+3\nu}}\right\}.
		\end{equation}
		This matches the lower bound in \cite{nbmirovskii_optimal_1985,nemirovsky_problem_1983}. The convex case $\mu = 0$ has been obtained by the methods in \cite{kamzolov_universal_2019,nbmirovskii_optimal_1985,nesterov_universal_2015}, and the restarted schemes in \cite{kamzolov_universal_2019,roulet_sharpness_2017} attained the complexity bound for $\mu>0$. Besides, Guminov et al. \cite{guminov_primal-dual_2018} obtained \cref{eq:Un-nu} for nonconvex problems, with an additional 1D line search.
	\end{itemize}
\end{rem}
\begin{rem}
	Then let us focus on the affine constraint case: $A \neq O$.
	\begin{itemize}		
		\item The Lipschitzian case $\nu = 1$:
		\begin{equation}\label{eq:A-mu-0}
			\min\left\{\frac{\nm{A}}{\epsilon}+\sqrt{\frac{L_h}{\epsilon}},\quad
			\frac{\nm{A}}{\sqrt{\mu \epsilon}}+\sqrt{\frac{L_h}{\mu}}\cdot\snm{\ln\epsilon} \right\}.
		\end{equation}
		This coincides with the lower complexity bound in \cite{ouyang_lower_2021}.
		The methods in \cite{chen_optimal_2014,nesterov_smooth_2005,xu_iteration_2021} achieved the bound for convex case $\mu=0$, and the strongly convex case $\mu>0$ can be found in \cite{xu_iteration_2021}. 
		\item The H\"{o}lderian case $0\leq \nu < 1$:
		\[
		\min\left\{\frac{\nm{A}}{\epsilon}+\left(\frac{M_\nu(h)}{\epsilon}\right)^\frac{2}{1+3\nu},\quad
		\frac{\nm{A}}{\sqrt{\mu \epsilon}}+	\left( \frac{M_\nu(h)}{\mu}\right)^{\frac{2}{1+3\nu}}\cdot \left(\frac{\mu}{\epsilon}\right)^{\frac{1-\nu}{1+3\nu}} \right\}.
		\]
		Similarly with \cref{eq:A-mu-0}, this universal mixed-type estimate has optimal dependence on $\nm{A}$ (corresponding to the affine constraint), and the remainder agrees with \cref{eq:Un-nu}, which is optimal with respect to $\mu$ and $M_\nu(h)$ (related to the objective $f$).  
	\end{itemize}		
\end{rem}

\section{Proof of \cref{lem:diff-Lk}}
\label{sec:pf-lem-diff-Lk}
Let us start from the difference 
$	\mathcal E_{k+1}-\mathcal E_{k} = \mathbb I_1+\mathbb I_2+\mathbb I_3$, where
\[
\left\{
\begin{aligned}
	\mathbb I_1	:={}& \mathcal L\left(x_{k+1},\lambda^*\right)
	-\mathcal L\left(x_{k},\lambda^*\right),\\
	\mathbb I_2:=	{}&\gamma_{k+1}
	D_\phi(x^*,v_{k+1})- 
	\gamma_{k}
	D_\phi(x^*,v_k),\\
	\mathbb I_3:={}&\frac{\beta_{k+1}}{2}
	\nm{\lambda_{k+1}-\lambda^*}^2 - 
	\frac{\beta_k}{2}
	\nm{\lambda_{k}-\lambda^*}^2.
\end{aligned}
\right.
\]
Notice that $x_k,\,x_{k+1}\in Q$ and the first term is easy to handle:
\begin{equation}\label{eq:I1}
	\mathbb I_1 = f(x_{k+1})-f(x_k)+\langle \lambda^*,A(x_{k+1}-x_k) \rangle.
\end{equation}
We derive the estimate of $\mathbb I_2$ in \cref{sec:I2} and finish the proof of \cref{eq:diff-Lk} in \cref{sec:I3-}.
\subsection{Estimate of $\mathbb I_2$}
\label{sec:I2}
Invoking the three-term identity \cref{eq:3-term-id} and the difference equation of $\{\gamma_k\}_{k\in \mathbb N}$ in \cref{eq:gama-betak}, we split the second term $\mathbb I_2$ as follows
\[
\begin{aligned}
	\mathbb I_2={}&(\gamma_{k+1}-\gamma_{k})D_\phi(x^*,v_{k+1})+\gamma_{k}\left[	D_\phi(x^*,v_{k+1})-	D_\phi(x^*,v_k) \right]\\
	= {}&\alpha_k(\mu-\gamma_{k+1})D_\phi(x^*,v_{k+1})
	-\gamma_{k}D_\phi(v_{k+1},v_k)\\
	{}&\quad
	+\gamma_{k}\dual{\nabla\phi(v_{k+1})-\nabla\phi(v_k),v_{k+1}-x^*}.
\end{aligned}
\]
Let us prove 
\begin{equation}\label{eq:vk1-vk}
	\begin{aligned}
		{}&\mu\alpha_k D_\phi(x^*,v_{k+1})
		+\gamma_{k}\dual{\nabla\phi(v_{k+1})-\nabla\phi(v_k),v_{k+1}-x^*}
		\\
		\leq{}		& h(x_k)-h(y_k)
		- \alpha_k\left[h(y_k)-h(x^*)+\big\langle \whk, Av_{k+1} -b\big\rangle\right]\\
		{}&\quad	-\alpha_k\left[g(v_{k+1})-g(x^*)+\dual{\nabla h(y_k),v_{k+1} - v_k}\right],
	\end{aligned}
\end{equation}
which leads to the desired estimate of $\mathbb I_2$:
\begin{equation}\label{eq:I3}
	\begin{aligned}
		\mathbb I_2\leq&-\alpha_k \gamma_{k+1}D_\phi(x^*,v_{k+1})	-\gamma_{k}D_\phi(v_{k+1},v_k)-\alpha_k \big\langle \whk, Av_{k+1} -b\big\rangle\\
		&\quad	-\alpha_k\left[g(v_{k+1})-g(x^*)+h(y_k)-h(x^*)\right]\\
		&\quad\quad +h(x_k)-h(y_k)		-\alpha_k\dual{\nabla h(y_k),v_{k+1} - v_k}.
	\end{aligned}
\end{equation}

To do this, define $\zeta_{k+1}$ by that
\begin{equation}\label{eq:vk1-vk-}
	\begin{aligned}
		{}&\gamma_k\big[\nabla\phi(v_{k+1})-\nabla\phi(v_k)\big] \\
		={}& \mu\alpha_k\big[\nabla\phi(y_k)-\nabla\phi(v_{k+1})\big]
		-\alpha_k\big[\nabla h(y_k)+\zeta_{k+1}+A^\top\whk\big].
	\end{aligned}
\end{equation}
Observing \eqref{eq:pc-imex-v}, it follows that $\zeta_{k+1}\in\partial g(v_{k+1})+N_Q(v_{k+1})$ and 
\[
\begin{aligned}
	{}&	-\alpha_k\big\langle\zeta_{k+1}, v_{k+1} - x^{*}\big\rangle
	\leq-\alpha_k\left[g(v_{k+1})-g(x^*)\right].
\end{aligned}
\]
Thanks to \cref{eq:3-term-id}, we have the decomposition
\[
\begin{aligned}
	{}&\mu\alpha_k	\dual{\nabla\phi(y_k)-\nabla\phi(v_{k+1}), v_{k+1} - x^{*}}\\
	={}&\mu\alpha_k
	\big[D_\phi(x^*,y_k)
	-D_\phi(x^*,v_{k+1})
	-D_\phi(v_{k+1},y_k)\big],
\end{aligned}
\]
and invoking \eqref{eq:pc-imex-y} leads to
\begin{align*}
	&
	-\alpha_k \dual{ \nabla h(y_k), v_{k+1} - x^{*}} \\
	=&-\alpha_k\dual{\nabla h(y_k),v_{k+1} - v_k}
	- \dual{ \nabla h(y_k), y_{k} - x_{k}} 
	- \alpha_k\dual{  \nabla h(y_k), y_{k} - x^{*}}.
\end{align*}
Since $x_k,\,y_k\in Q$, by \cref{assum:h-g} we obtain
\begin{align*}
	{}&
	- \dual{ \nabla h(y_k), y_{k} - x_{k}} 
	- \alpha_k\dual{  \nabla h(y_k), y_{k} - x^{*}}\\
	\leq  {}&h(x_k)-h(y_k)
	- \alpha_k\left[h(y_k)-h(x^*)+\mu D_\phi(x^*,y_k)\right)].
\end{align*}
Hence, combining the above estimates with \cref{eq:vk1-vk-} proves \cref{eq:vk1-vk}.
\subsection{Proof of  \cref{eq:diff-Lk}}
\label{sec:I3-}
Similarly as before, by \cref{eq:3-term-id}, \eqref{eq:pc-imex-l} and \cref{eq:gama-betak}, the third term $\mathbb I_3$
is rearranged by that
\[
\begin{aligned}
	\mathbb I_3=	{}&	 -\frac{\alpha_k\beta_{k+1}}{2}\nm{\lambda_{k+1}-\lambda^*}^2
	-\frac{\beta_{k}}{2}\big\|\lambda_{k+1}-\lambda_k\big\|^2\\
	{}&\qquad		+\alpha_{k}\big\langle Av_{k+1}-b,\lambda_{k+1}-\lambda^*\big\rangle.
\end{aligned}
\]
To match the cross term $		-\alpha_k \big\langle \whk, Av_{k+1} -b\big\rangle$ in the estimate of $\mathbb I_2$ (cf.\cref{eq:I3}), we rewrite the last term as follows
\[
\begin{aligned}
	{}&
	\alpha_{k}\big\langle Av_{k+1}-b,\lambda_{k+1}-\lambda^*\big\rangle\\
	={}&\alpha_{k}\big\langle Av_{k+1}-b,\lambda_{k+1}-\whk\big\rangle
	+\alpha_{k}\big\langle Av_{k+1}-b,\whk-\lambda^*\big\rangle.
\end{aligned}
\]
In view of \cref{eq:3-term-id-quad} and \eqref{eq:pc-imex-l}, we get
\[
\begin{aligned}
	{}&\alpha_{k}\big\langle Av_{k+1}-b,\lambda_{k+1}-\whk\big\rangle=\beta_k\dual{\lambda_{k+1}-\lambda_{k},\lambda_{k+1}-\whk}\\
	={}&
	\frac{\beta_{k}}{2}\big\|\lambda_{k+1}-\lambda_k\|^2
	+	\frac{\beta_{k}}{2} \big\|\lambda_{k+1}-\whk\big\|^2
	-\frac{\beta_{k}}{2}\big\|\lambda_{k}-\whk\big\|^2,
\end{aligned}
\]
which gives
\[
\mathbb I_3\leq  -\frac{\alpha_k\beta_{k+1}}{2}\nm{\lambda_{k+1}-\lambda^*}^2
+\frac{\beta_{k}}{2}\big\|\lambda_{k+1}-\whk\big\|^2
+\alpha_{k}\big\langle Av_{k+1}-b,\whk-\lambda^*\big\rangle.
\]
Therefore, collecting this with \cref{eq:I1,eq:I3} yields
\[
\begin{aligned}
	\mathcal E_{k+1}-\mathcal E_{k}\leq  &
	-\alpha_k\mathcal E_{k+1}
	+\frac{\beta_{k}}{2}\big\|\lambda_{k+1}-\whk\big\|^2
	-\gamma_{k}D_\phi(v_{k+1},v_k)
	\\
	{}&		\quad+(1+\alpha_k)\left[h(x_{k+1})- h(y_{k})\right]	-\alpha_k\dual{\nabla h(y_k),v_{k+1} - v_k}\\
	{}&	\qquad			+(1+\alpha_k)g(x_{k+1})-g(x_{k})	-\alpha_kg(v_{k+1}).		
\end{aligned}
\]
From \eqref{eq:pc-imex-x}, we see that $x_{k+1}$ is a convex combination of $x_k$ and $v_{k+1}$, which implies
\[
(1+\alpha_k)g(x_{k+1})\leq g(x_{k})	+\alpha_kg(v_{k+1}).
\]
Thanks to \cref{eq:xk1-yk} and the relation $\alpha_k(v_{k+1}-v_k)=(1+\alpha_k)	(x_{k+1}-y_k) $ (cf.\eqref{eq:pc-imex-y} and \eqref{eq:pc-imex-x}), we obtain
\[
\begin{aligned}
	{}	&(1+\alpha_k)\left[h(x_{k+1})- h(y_{k})\right]	-\alpha_k\dual{\nabla h(y_k),v_{k+1} - v_k}\\
	\leq{}&\frac{\alpha_k^2M_{k+1}}{2+2\alpha_k}\nm{v_{k+1}-v_k}^2
	+\frac{\delta_{k+1}}{2}(1+\alpha_k).
\end{aligned}
\]
Consequently, applying \cref{eq:D-phi} leads to
\[
\begin{aligned}
	\mathcal E_{k+1}-\mathcal E_{k}\leq  &
	-\alpha_k\mathcal E_{k+1}+\frac{\delta_{k+1}}{2}(1+\alpha_k)
	+\frac{\beta_{k}}{2}\big\|\lambda_{k+1}-\whk\big\|^2\\
	&\quad
	+\frac{\alpha_k^2M_{k+1}-\gamma_{k}(1+\alpha_k)}{1+\alpha_k}D_\phi(v_{k+1},v_k).
\end{aligned}
\]
Recall that $\whk = \lambda_k+\alpha_k/\beta_k(Av_k-b)$, which together with \eqref{eq:pc-imex-l} gives $\lambda_{k+1}-	\whk =\alpha_k/\beta_k A(v_{k+1}-v_k)$ and 
\[
\begin{aligned}
	\mathcal E_{k+1}-\mathcal E_{k}\leq  {}&
	\!-\alpha_k\mathcal E_{k+1}+\frac{\delta_{k+1}}{2}(1+\alpha_k)\\
	{}&\quad	+\frac{1}{\beta_{k}}
	\left[\alpha_k^2(\beta_{k+1}M_{k+1}
	+\nm{A}^2)-\gamma_k\beta_k\right]
	D_\phi(v_{k+1},v_k).
\end{aligned}
\]
Since $\beta_{k+1}\leq \beta_{k}$ (cf.\cref{eq:gama-betak}), the desired estimate \cref{eq:diff-Lk} follows immediately from \cref{eq:ak}. This finishes the proof of \cref{lem:diff-Lk}.

\section{Proof of \cref{lem:est-tk}}
\label{sec:pf-lem-est-tk}
The key to complete the proof of \cref{lem:est-tk} is the difference inequality \cref{eq:diff-bk}.
In \cref{sec:est-yt}, we shall introduce an auxiliary differential inequality (cf.\cref{eq:yt}) that can be viewed as a continuous analogue  to \cref{eq:diff-bk}. Later in \cref{sec:mu-0,sec:mu>0}, we finish the proofs of \cref{eq:bk-mu-0,eq:bk-mu>0} by using the asymptotic estimate of \cref{eq:yt}.
\subsection{A differential inequality}
\label{sec:est-yt}
Let $\eta,\,R\geq 0$ and $\theta>1$ be real constants such that $\eta\leq \theta-1$.
Assume $y\in W^{1,\infty}(0,\infty)$ is positive and satisfies the differential inequality
\begin{equation}\label{eq:yt}
	y'(t) \leq  -\frac{\sigma(t) y^\theta(t)}{\sqrt{\varphi(t)y^{2\eta}(t) + R^2}},\quad y(0) = 1,
\end{equation}
where $\sigma\in L^1(0,\infty)$ is nonnegative, and $\varphi\in C^1[0,\infty)$ is positive and nondecreasing. Plugging the trivial estimate 
\[
\sqrt{\varphi(t)y^{2\eta}(t) + R^2}\leq \sqrt{\varphi(t)}y^{\eta}(t)+R
\]
into \cref{eq:yt} gives
\begin{equation}\label{eq:Ly'}
	\left(\frac{\sqrt{\varphi(t)}}{y^{\theta-\eta}(t)}+\frac{R}{y^{\theta}(t)}\right)y'(t)\leq -\sigma(t).
\end{equation}

The decay estimate of $y(t)$ is given below. Detailed proof can be found in \cref{sec:app-pf}.
\begin{lem}\label{lem:est-y}
	Assume $y\in W^{1,\infty}(0,\infty)$ is positive and satisfies \cref{eq:yt}. Then for all $t>0$, we have
	\[
	y(t)\leq C_{\theta,\eta}\left\{
	\begin{aligned}
		{}&\left(\frac{\sqrt{\varphi(t)}}{\Sigma(t)}\right)^{\frac{1}{\theta-\eta-1}}
		+\left(\frac{R}{\Sigma(t)}\right)^{\frac{1}{\theta-1}}
		&&\text{if }\eta<\theta-1,\\
		{}&\exp\left(-\frac{\Sigma(t)}{2\sqrt{\varphi(t)}}\right)+
		\left(\frac{R}{\Sigma(t)}\right)^{\frac{1}{\eta}}&&\text{if }
		\eta=\theta-1,
	\end{aligned}
	\right.
	\]
	where $\Sigma(t) :=\int_{0}^{t}\sigma(s)\dd s$ and $C_{\theta,\eta}>0$ depends only on $\theta$ and $\eta$.
\end{lem}
\subsection{Proof of \cref{eq:bk-mu-0}}
\label{sec:mu-0}
In this case, by \cref{eq:gama-betak}, we have $\gamma_k = \gamma_0\beta_{k}$ and \cref{eq:diff-bk} becomes 
\begin{equation}\label{eq:diff-lk-mu-0}
	\beta_{k+1}-\beta_{k} \leq  -\frac{\sqrt{\gamma_{0}}\beta_k\beta_{k+1}}{\sqrt{2\sqrt{2}\beta_kM(\nu,\delta_{k+1})
			+\nm{A}^2}},
\end{equation}
where $M(\nu,\delta_{k+1}) =\delta_{k+1}^{\frac{\nu-1}{\nu+1}}[M_\nu(h)]^{\frac{2}{1+\nu}} $ with $\delta_{k+1} = \beta_{k+1}/(k+1)$.

Define a piecewise continuous linear interpolation
\begin{equation}\label{eq:y}
	y(t): ={}\beta_{k}(k+1-t)+\beta_{k+1}(t-k)\quad\forall\,t\in[k,k+1),
	\, k\in\mathbb N.
\end{equation}
Clearly, $y\in W^{1,\infty}(0,\infty)$ is positive and $0<y(t)\leq y(0) = 1$. In particular, we have $\beta_k = y(k)$  for all $k\in\mathbb N$, and the decay estimate of $\beta_k$ is transferred into the asymptotic behavior of $y(t)$, which satisfies
\begin{equation}\label{eq:bk-yt-mu-0}
	y' (t)\leq  -\frac{\sqrt{\gamma_0}/2\,y^{2}(t)}{\sqrt{8\sqrt{2}\varphi(t)[y(t)]^{\frac{2\nu}{1+\nu}}
			+\nm{A}^2 }},
\end{equation}
where $\varphi(t) := (t+1)^{\frac{1-\nu}{1+\nu}}[M_\nu(h)]^{\frac{2}{1+\nu}}$. Thus, utilizing \cref{lem:est-y} gives
\[
\beta_k=y(k)\leq C_\nu\left(
\frac{\nm{A}}{\sqrt{\gamma_0}k} + \frac{M_\nu(h)}{\gamma_0^{\frac{1+\nu}{2}}k^{\frac{1+3\nu}{2}}}
\right)\quad\forall\,k\geq 1,
\]
where $C_\nu>0$ depends only on $\nu$. This establishes \cref{eq:bk-mu-0}.

Below, let us verify \cref{eq:bk-yt-mu-0}. Since $\gamma_k\leq \max\{\gamma_0,\mu\}\leq\nm{A}^2$, from \cref{eq:ak} we find that 
\[
\alpha_k\leq \sqrt{\gamma_k\beta_k}/\nm{A}
\leq 1\quad\forall\,k\in\mathbb N.
\]
For any $t\in(k,k+1)$, it is clear that 
\[
1\geq \frac{\beta_{k+1}}{y(t)}\geq \frac{\beta_{k+1}}{\beta_{k}}=\frac{1}{1+\alpha_k}\geq \frac{1}{2},\quad\text{and}\quad 1\leq  \frac{\beta_k}{y(t)}\leq \frac{\beta_{k}}{\beta_{k+1}}\leq2,
\]
which implies 
\[
\beta_kM(\nu,\delta_{k+1})= \varphi(k)\beta_k\beta_{k+1}^{\frac{\nu-1}{\nu+1}}
\leq 2^{\frac{2}{1+\nu}}\varphi(t)[y(t)]^{\frac{2\nu}{1+\nu}}.
\]
Since $y'(t) = \beta_{k+1}-\beta_k$, plugging the above estimate into \cref{eq:diff-lk-mu-0} proves \cref{eq:bk-yt-mu-0}.
\subsection{Proof of \cref{eq:bk-mu>0}}
\label{sec:mu>0}
By \cref{eq:gama-betak}, we have $\gamma_{k}\geq \gamma_{\min}  = \min\{\gamma_0,\mu\}$, and \cref{eq:diff-bk} becomes 
\[
\beta_{k+1}-\beta_{k} \leq  -\frac{\sqrt{\gamma_{\min}\beta_k}\beta_{k+1}}{\sqrt{2\sqrt{2}\beta_kM(\nu,\delta_{k+1})
		+\nm{A}^2}}.
\]
Recall the piecewise interpolation $y(t)$ defined by \cref{eq:y}. Similarly with \cref{eq:bk-yt-mu-0}, we claim that 
\[
y' (t)\leq  -\frac{\sqrt{\gamma_{\min}}/2\,y^{3/2}(t)}{\sqrt{8\sqrt{2}\varphi(t)[y(t)]^{\frac{2\nu}{1+\nu}}
		+\nm{A}^2 }},
\]
and invoking \cref{lem:est-y} again gives
\[
\beta_k\leq C_\nu\left\{
\begin{aligned}
	{}&\frac{\nm{A}^2}{\gamma_{\min} k^2}
	+\frac{[M_\nu(h)]^{\frac{2}{1-\nu}}}{\gamma_{\min}^{\frac{1+\nu}{1-\nu}}k^{\frac{1+3\nu}{1-\nu}}}&&\text{ if }\nu<1,\\
	{}&	 \frac{\nm{A}^2}{\gamma_{\min} k^2}+\exp\left(-\frac{k}{8\sqrt{3}}\sqrt{\frac{\gamma_{\min}}{L_h}}\right)&&\text{ if }\nu=1,
\end{aligned}
\right.
\]
which proves \cref{eq:bk-mu>0} and completes the proof of \cref{lem:est-tk}.

\section{Numerical Examples}
\label{sec:num}
In this part, we provide several numerical tests to validate the performance of our \cref{algo:UAPD} (denoted shortly by UAPD). It is compared with Nesterov's FGM \cite{nesterov_universal_2015} and the AccUniPDGrad method  \cite{yurtsever_universal_2015}, respectively for unconstrained and affinely constrained problems.

Both UAPD and FGM involve the proximal mapping of the nonsmooth part $g$ under Bregman distance. However, FGM performs one more proximal calculation for updating $v_k$, and in line search part, FGM and AccUniPDGrad use the tolerance $\delta_k = \epsilon\tau_k$ with $\tau_k = \mathcal O(1/k)$, which is smaller than ours $\delta_k = \beta_k/k$. As discussed previously in \cref{rem:compare-line-search}, this will lead to over-estimate issue, especially for H\"{o}lderian case (cf. \cref{sec:mat-game}) and smooth problems with large Lipschitz constants (cf. \cref{sec:regu-mat-game}).
\subsection{Matrix game}
\label{sec:mat-game}
The problem reads as
\begin{equation}\label{eq:matrix-game}
	\min_{x\in \Delta_n}	\max_{y\in \Delta_m}\dual{x,Py} = 
	\min_{x\in \Delta_n}\left\{h(x): = 	\max_{1\leq j\leq m}\dual{p_j,x}\right\},
\end{equation}
where $P =(p_1,p_2,\cdots,p_m)\in\R^{n\times m}$ is the given payoff matrix and $\Delta_{\times }$ denotes the standard simplex with $\times =m$ or $n$. According to von Neumann’s minimax theorem \cite[Corollary 15.30]{Bauschke2011}, it is also equivalent to 
\[
\begin{aligned}
	\max_{y\in \Delta_m}\left\{\min_{1\leq i\leq n}\dual{e_i,Py}\right\} ={}& -\min_{y\in \Delta_m}\left\{-\min_{1\leq i\leq n}\dual{e_i,Py}\right\}\\
	={}&	-\min_{y\in \Delta_m}\left\{g(y) := \max_{1\leq i\leq n}\dual{q_i,y}\right\},
\end{aligned}
\]
where $P^\top =- (q_1,q_2,\cdots,q_n)\in\R^{m\times n}$. As we do not the know the optimal value of $h$ and $g$, it is more convenient to  consider 
\begin{equation}\label{eq:min-f-g}
	\min_{x\in \Delta_n,\,y\in \Delta_m}\left\{f(x,y): = h(x)+g(y)\right\}.
\end{equation}
Clearly, this problem is nonsmooth ($f$ is only Lipschitz continuous) and the minimal value is zero. A natural prox-function for this problem is the entropy $\phi(x) = \dual{x,\ln x}$.

\begin{figure}[H]
	\centering
	\includegraphics[width=10.5cm]{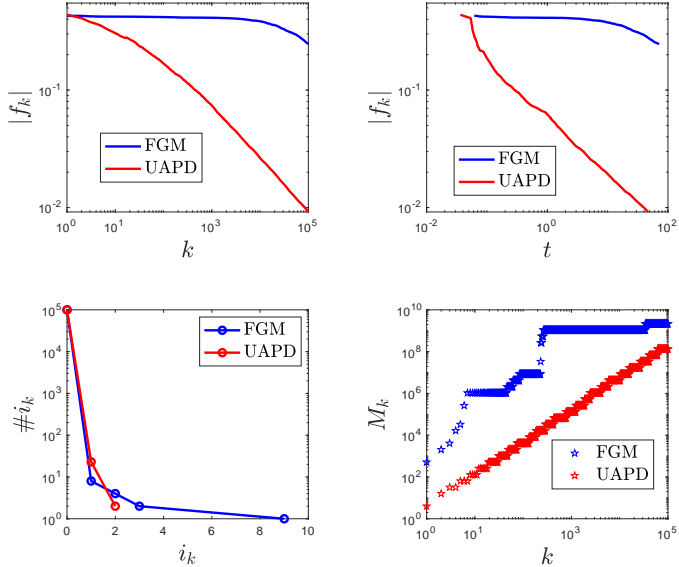}
	\caption{Numerical performances of FGM and UAPD on the matrix game problem with $m = 100,\,n = 400$.}
	\label{fig:matrix-game-randn}
\end{figure}
We record (i) the decay behavior of the objective residual $|f_k| = |f(x_k)|$ (with respect to iteration number $k$ and running time $t$ in seconds), (ii) the total number $\#i_k$ of the line search step $i_k$, and (iii) the approximate Lipschitz constant $M_k$. The pay off matrix $P$ is generated from normal distribution and for FGM, we set the accuracy parameter $\epsilon = $1e-5. 

Numerical results are displayed in \cref{fig:matrix-game-randn}, from which we see that our UAPD outperforms FGM, with faster convergence and smaller Lipschitz constants. The total number $\#i_k$ is close to each other. But, FGM produces over-estimated Lipschitz parameters with dramatically growth behavior since it adopts smaller  tolerance $\epsilon/k$ for line search procedure.

\begin{figure}[H]
	\centering
	\includegraphics[width=12.5cm]{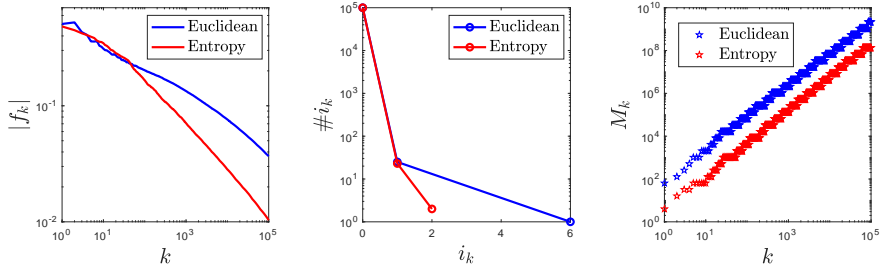}
	\caption{Numerical results of UAPD on the matrix game problem with different prox-functions.}
	\label{fig:matrix-game-prox-randn}
\end{figure}
Besides, we investigate the difference between Euclidean distance $\phi(x) = 1/2\nm{x}^2$ and entropy function $\phi(x) = \dual{x,\ln x}$. It is observed that these two cases are very similar in line search procedure but entropy function leads to better convergence rate. 
\begin{figure}[H]
	\centering	
	\includegraphics[width=10.5cm]{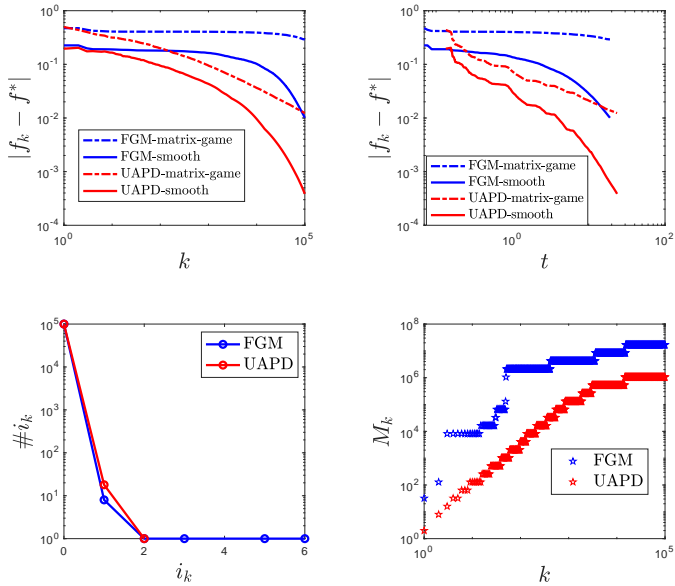}
	\caption{Numerical performances of FGM and UAPD on regularized matrix game with $m = 100,\,n = 400$.}
	\label{fig:matrix-game-smooth-randn}
\end{figure}
\subsection{Regularized matrix game problem}
\label{sec:regu-mat-game}
The problem \cref{eq:matrix-game} admits an approximation
\begin{equation}\label{eq:regu-f}
	f_\sigma(x): = \sigma\ln\left(\sum_{j=1}^{m}e^{\dual{p_j,x}/\sigma}\right),
\end{equation}
where $\sigma>0$ denotes the smoothing parameter.  This regularized objective is smoother than the original one. According to \cite[Eq.(4.8)]{nesterov_smooth_2005}, we choose $\sigma = \epsilon/(2\ln m)$, and the Lipschitz constant of $\nabla f_\sigma$ is $L_\sigma =\max_{i,j}|P_{i,j}|^2/(4\sigma)$.

We then apply UAPD and FGM (with $\epsilon = $1e-5) to the smooth problem \cref{eq:regu-f} and report the numerical outputs in \cref{fig:matrix-game-smooth-randn}. The optimal value $f^* $ is obtained by running UAPD with enough iterations. Similarly as before, our UAPD is superior to FGM in convergence and approximate Lipschitz constant. Also, we plot the objective residuals of the original matrix game and find that with smoothing technique both two methods perform better than before.

\subsection{Continuous Steiner problem}
Let us consider one more unconstrained problem
\begin{equation}\label{eq:min-stein}
	\min_{x\in \R^n_+}\,f(x) = \sum_{j=1}^{m}\nm{x-a_j},
\end{equation}
where $a_j\in\R^n$ denotes a given location. Note that the objective is actually quite smooth far away from each location $a_j$. We generate $a_j$ from normal distribution and run UAPD with enough iterations to obtain an approximated optimal value $f^*$. Numerical results in \cref{fig:Steiner-rand} show that both FGM (with $\epsilon = $1e-8) and UAPD work well and possess similar convergence behaviors. Moreover, as $\nabla f$ is almost Lipschitz continuous and the magnitude of the Lipschitz constant $L$ is not so large, the over-estimated issue of FGM is negligible, and the approximated constant $M_k$ is the same as that of UAPD.

\begin{figure}[H]
	\centering
	\includegraphics[width=10.5cm]{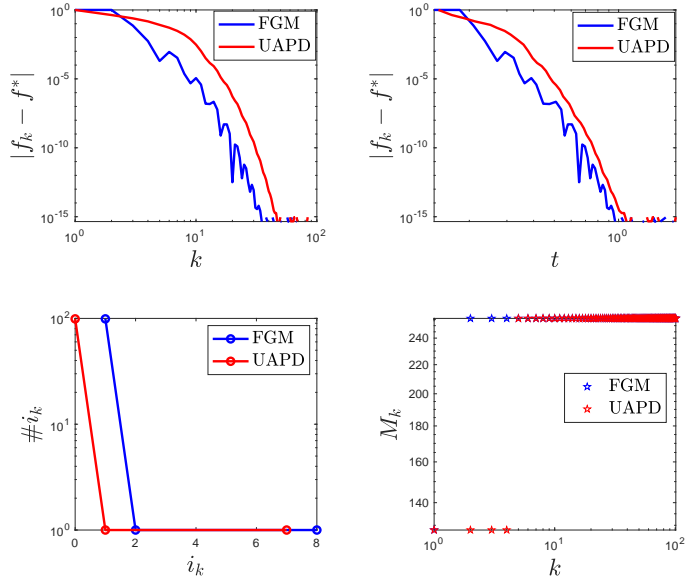}
	\caption{Numerical performances of FGM and UAPD on the continuous Steiner problem with $m = 800,\,n = 400$.}
	\label{fig:Steiner-rand}
\end{figure}
\subsection{Basis pursuit problem}
In the last example, we move to the basis pursuit problem
\[
\min_{x\in \R^n}\nm{x}_1\quad{\rm s.t.}\,Ax=b,
\]
where $A\in \R^{m\times n}$ and $b\in\R^m$. To be compatible with the problem setting of AccUniPDGrad, we consider an equivalent formulation
\[
\min_{x\in \R^n}\frac{1}{2}\nm{x}_1^2\quad{\rm s.t.}\,Ax=b.
\]
The dual problem reads as 
\[
\min_{\lambda\in \R^{m}}\,\left\{\varphi(\lambda): = \dual{b,\lambda}+\frac{1}{2} \big\|A^\top \lambda\big\|_\infty^2\right\}.
\]

Note that existing accelerated Bregman method \cite{huang_accelerated_2013} and accelerated ALM \cite{xu_accelerated_2018} can be applied to this problem with theoretical rate $\mathcal O(1/k)$. But we only focus on the comparison between UAPD and AccUniPDGrad \cite{yurtsever_universal_2015}, as black-box type methods with line search procedure. We mention that the AccUniPDGrad method  also uses smaller tolerance $\epsilon/k$ as that in FGM.
Numerical results are showed in \cref{fig:basis}, which indicate that (i) our UAPD has smaller objective residual and feasibility violation, and (ii) the line search procedure is more efficient with smaller total number $\#i_k$ and Lipschitz constant $M_k$. 
\begin{figure}[H]
	\centering
	\includegraphics[width=13cm]{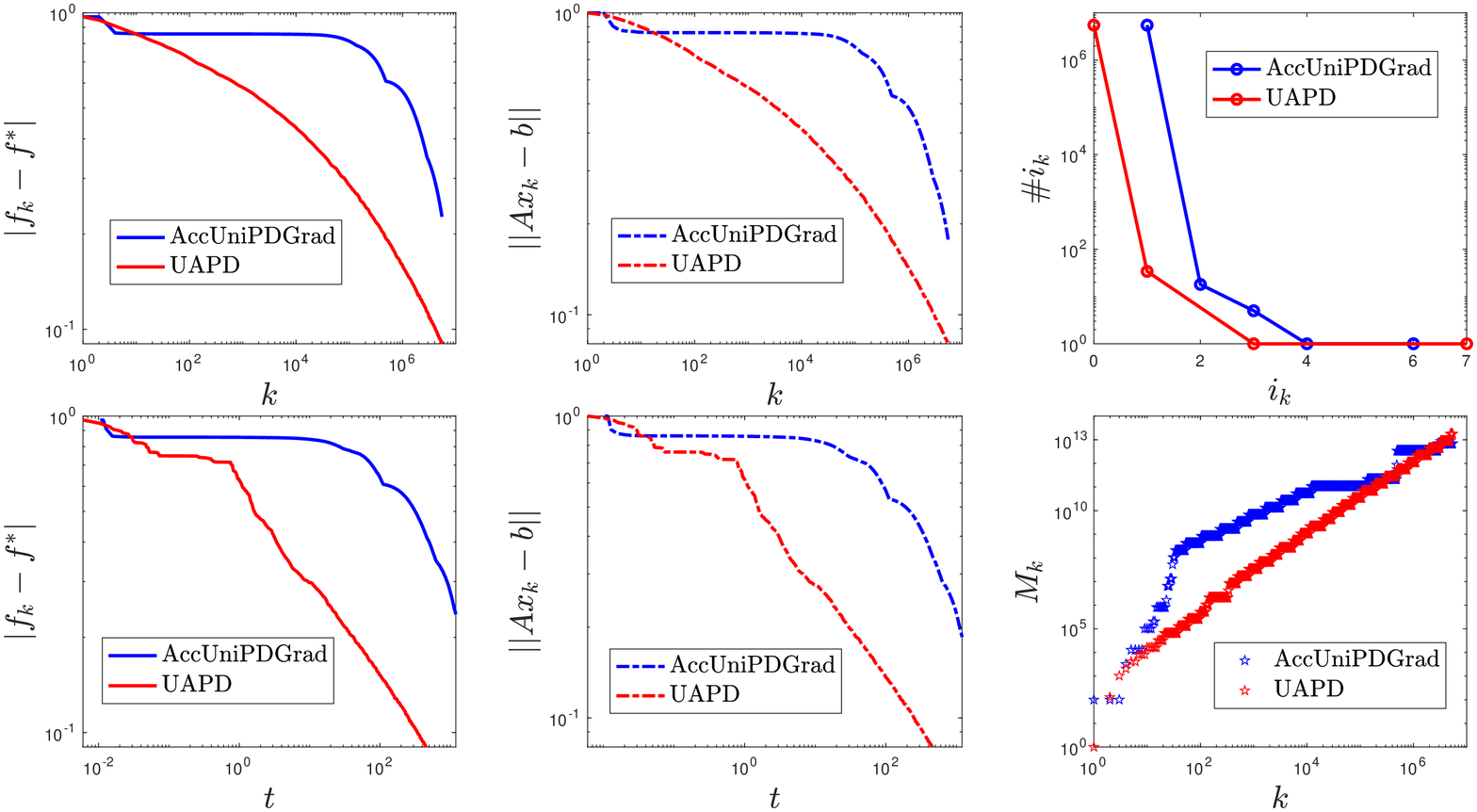}
	\caption{Numerical performances of AccUniPDGrad and UAPD on the basis pursuit problem with $m = 100,\,n = 500$. The desired accuracy for AccUniPDGrad is $\epsilon = 1$e-3.}
	\label{fig:basis}
\end{figure}

\appendix
\section{Proof of \cref{lem:Mk}}
\label{sec:app-Mk}
Let us first prove \cref{eq:Mk}. Recall that $i_k$ is the smallest integer such that 
\[
h(x_{k,i_k})-\Delta_{k,i_k}\leq\frac{\delta_{k,i_{k}}}{2}.
\]
If $i_k = 0$, then $M_{k+1} = M_{k}$. If $i_k\geq 1$, then we claim that 
\begin{equation}\label{eq:bd-Mkik}
	M_{k,i_k}\leq 2M(\nu,\delta_{k,i_{k}-1}).
\end{equation}
Otherwise, we have $M_{k,i_{k}-1} = M_{k,i_k}/2>M(\nu,\delta_{k,i_{k}-1})$. According to \cref{prop:M-delta}, this implies immediately that 
\[
h(x_{k,i_{k}-1})-\Delta_{k,i_{k}-1}\leq \frac{\delta_{k,i_{k}-1}}{2},
\]
which yields a contradiction and thus verifies the estimate \cref{eq:bd-Mkik}. Additionally, by \cref{eq:ak-ik}, we have $\alpha_{k,i_k}\leq \alpha_{k,i_{k}-1}\leq \sqrt{2}\alpha_{k,i_{k}}$. Thus, using \cref{eq:M-nu-delta,eq:deltak-ik,eq:bd-Mkik} leads to
\begin{equation}\label{eq:bd-Mkik-1}
	M_{k+1} = M_{k,i_k}\leq 2\sqrt{2}M(\nu,\delta_{k,i_{k}})
	=2\sqrt{2}M(\nu,\delta_{k+1}).
\end{equation}
This implies that for all $k\geq 0$, we have
\begin{equation}\label{eq:Mk-bd-1}
	M_{k+1}\leq 	\max\left\{2\sqrt{2}M(\nu,\delta_{k+1}), \,M_{k}\right\}.
\end{equation}
Note that $\delta_k = \beta_k/k$ and $\beta_k$ is decreasing. Thus $\delta_{i}\geq \delta_{k}$ and $M(\nu,\delta_{i})\leq M(\nu,\delta_{k+1})$ for all $1\leq i\leq k$, this indicates that 
\[
M_{k+1}\leq 	\max\left\{2\sqrt{2}M(\nu,\delta_{k+1}), M_{i}\right\}\quad\forall\,1\leq i\leq k.
\]
Taking $i=1$ and using \cref{eq:Mk-bd-1}  with $k=0$, we obtain
\[
M_{k+1}\leq 	\max\left\{ 2\sqrt{2}M(\nu,\delta_{k+1}), M_{0}\right\},
\]
which proves \cref{eq:Mk}.

Then, let us verify \cref{eq:bd-ik}. Observing that 
\[
M_{k+1} = M_{k,i_k} = 2^{i_k-1}M_{k}\quad\Longrightarrow\quad i_k = 1+\log_2\frac{M_{k+1}}{M_{k}},
\]
we get
\[
\sum_{j=0}^{k}i_j= k+1+	\log_2\frac{M_{k+1}}{M_{0}}\leq k+1+\max\left\{1, \log_2\frac{M(\nu,\delta_{k+1})}{M_{0}/2\sqrt{2}}\right\}.
\]
Since $M(\nu,\delta_{k+1}) = \delta_{k+1}^{\frac{\nu-1}{\nu+1}}[M_\nu(h)]^{\frac{2}{\nu+1}}$, we complete the proof of \cref{lem:Mk}.
\section{Accelerated Bregman Primal-Dual Flow}
\label{app:abpd}
Recall the conjugate function
\[
\phi^*(\xi) := \sup_{x\in Q}\left\{
\dual{\xi,x}-\phi(x)
\right\}\quad\forall\,\xi\in\R^n.
\]
We have the relation: $\xi=\nabla \phi(x)~\Longleftrightarrow~ x=\nabla \phi^*(\xi)$; see \cite[Theorem 16.23]{Bauschke2011}.
Therefore, by introducing $w = \nabla \phi(v)$, we obtain an alternative first-order formulation of \cref{eq:apd-sys-app}:
\begin{equation}\label{eq:apd-sys-app-equi}
	\left\{
	\begin{aligned}
		{}&x' = \nabla\phi^*(w)-x,\\
		{}&\gamma w'{}\in \mu (\nabla\phi(x)-w)-\big(\partial f(x)+N_Q(x)+A^\top \lambda\big),\\
		{}&\beta \lambda' {}= A\nabla\phi^*(w)-b.
	\end{aligned}
	\right.
\end{equation}
\subsection{Well-posedness and exponential decay}
Let us focus on the smooth setting: $Q = \R^n$ and $f\in C^1$ satisfies 
\[
f(x)\geq f(y)+\dual{\nabla f(y),x-y}+\mu D_\phi(x,y)\quad\forall\,x,\,y\in\R^n,
\]
with $\mu\geq 0$.
Then our ABPD flow dynamics \cref{eq:apd-sys-app} becomes 
\begin{subnumcases}{\label{eq:apd-sys-app-smooth}}
	\label{eq:apd-sys-app-smooth-x}
	{}	x' = v-x,\\
	\label{eq:apd-sys-app-smooth-v}		
	{}		\gamma \frac{\dd}{\dd t} \nabla \phi(v){}= \mu (\nabla\phi(x)-\nabla\phi(v))-\nabla f(x)-A^\top \lambda,\\
	\label{eq:apd-sys-app-smooth-l}
	{}		\beta \lambda' {}= Av-b.
\end{subnumcases}
By \cref{eq:apd-sys-app-equi}, this is also equivalent to 
\begin{equation}\label{eq:apd-sys-flow}
	\left\{
	\begin{aligned}
		{}&x' = \nabla\phi^*(w)-x,\\
		{}&\gamma w'{}= \mu (\nabla\phi(x)-w)-\nabla f(x)-A^\top \lambda,\\
		{}&\beta \lambda' {}= A\nabla\phi^*(w)-b.
	\end{aligned}
	\right.
\end{equation}

Recall that $\gamma$ and $\beta$ are governed by \cref{eq:gama-beta}, which actually admits explicit solutions
\[
\beta(t) = \beta_0e^{-t},\quad \gamma(t) = \mu + (\gamma_{0}-\mu)e^{-t}.
\]
Since $\phi$ is $1$-strongly convex (cf. \cref{eq:D-phi}), $\nabla\phi^*$ is $1$-Lipschitz continuous. Consequently, if both $\nabla f$ and $\nabla \phi$ are Lipschitz continuous, then by standard theory of ordinary differential equations, we conclude that the dynamical system \cref{eq:apd-sys-flow} admits a unique classical $C^1$ solution $(x,w,\lambda)$. This also promises that our ABPD flow \cref{eq:apd-sys-app-smooth} exists a unique solution $(x,v,\lambda)$ with $v = \nabla\phi^*(w)$ being continuous.

We then introduce a Lyapunov function
\begin{equation}\label{eq:Et}
	\mathcal E(x,v,\lambda) :=\mathcal L(x,\lambda^*)-\mathcal L(x^*,\lambda)+\gamma D_\phi(x^*,v)+\frac{\beta}{2}\nm{\lambda-\lambda^*}^2,
\end{equation}
which is a continuous analogue to the discrete one \cref{eq:Ek}. 
\begin{thm}
	\label{eq:exp-ABG}
	Let $(x,v,\lambda)\in C^1(\R_+;\R^n)\times C^0(\R_+;\R^n)\times C^1(\R_+;\R^m)$ be the unique solution to the ABPD flow \cref{eq:apd-sys-app-smooth}. Then we have
	\begin{equation}\label{eq:dt-E}
		\frac{\dd }{\dd t}			\mathcal E(x,v,\lambda) 
		\leq  
		-	\mathcal E(x,v,\lambda) -\mu 	D_\phi(v,x),
	\end{equation}
	which implies the exponential decay rate
	\[
	e^t\mathcal E(x(t),v(t),\lambda(t))  + \mu\int_0^te^{s}D_\phi(v(t),x(t))\dd s
	\leq  
	\mathcal E(x_0,v_0,\lambda_0),
	\]
	for all $t\geq 0$.
\end{thm}
\begin{proof}
	Taking the derivative with respect to the time variable gives
	\[
	\begin{aligned}
		\frac{\dd }{\dd t}\mathcal E(x,v,\lambda) = {}&\dual{\nabla_x\mathcal L(x,\lambda^*),x'} +
		\gamma'D_\phi(x^*,v)
		+\gamma\frac{\dd }{\dd t}D_\phi(x^*,v)\\
		{}&\quad +\frac{\beta'}{2}\nm{\lambda-\lambda^*}^2+\beta\dual{\lambda-\lambda^*,\lambda'}.
	\end{aligned}
	\]
	Since $w = \nabla\phi(v)\in C^1(\R_+;\R^n)$,  we see that $D_\phi(x^*,v)$ is continuous differentiable in terms of $t$ and by \eqref{eq:apd-sys-app-smooth-v}, we have
	\[
	\begin{aligned}
		{}&\gamma\frac{\dd }{\dd t}D_\phi(x^*,v)=\dual{\gamma\frac{\dd }{\dd t}\nabla\phi(v) ,v-x^*}\\
		={}&\mu\dual{\nabla \phi(x) - \nabla \phi(v)  ,v-x^*}
		-\dual{\nabla f(x) +A^\top\lambda,v-x^*}.
	\end{aligned}
	\]
	Then using the three-term identity \cref{eq:3-term-id} and following the proof of \cite[Lemma 2.1]{luo_acc_primal-dual_2021}, we can verify \cref{eq:dt-E} and complete the proof.
\end{proof}
\section{Proof of \cref{lem:est-y}}
\label{sec:app-pf}
\subsection{The case  $\eta=\theta-1$}
The estimate \cref{eq:Ly'} becomes
\begin{equation}\label{eq:y'-case1}
	\sqrt{\varphi(t)}\frac{y'(t)}{y(t)}+R\frac{y'(t)}{y^{\theta}(t)}\leq -\sigma(t).
\end{equation}
Since $y(0) =  1$ and $y'(t)\leq 0$, it holds that $0<y(t)\leq 1$ for all $t\geq 0$. As $\varphi(t)$ is positive and nondecreasing, we obtain
\[
\left(\sqrt{\varphi}\ln y\right)' = \frac{\varphi'}{2\sqrt{\varphi}}\ln y + \sqrt{\varphi} \frac{y'}{y}\leq \sqrt{\varphi} \frac{y'}{y}.
\]
Combining this with \cref{eq:y'-case1} gives
\[
\begin{aligned}
	\left(\sqrt{\varphi(t)}\ln y(t)+\frac{R}{1-\theta}y^{1-\theta}(t)\right)'\leq -\sigma(t),
\end{aligned}
\]
and integrating over $(0,t)$ leads to
\begin{equation}\label{eq:est-y-case1-mid}
	\sqrt{\varphi(t)}\ln \frac{1}{y(t)}+\frac{R}{\theta-1}\left(y^{1-\theta}(t)-1\right)\geq \int_{0}^{t}\sigma(s)\dd s=\Sigma(t).
\end{equation}

Define 
\begin{equation}\label{eq:Y1-Y2}
	Y_1(t) := {}\exp\left(-\frac{\Sigma(t)}{2\sqrt{\varphi(t)}}\right)
	\quad\text{and}\quad 
	Y_2(t) :={}  \left(1+\frac{\theta-1}{2R}\Sigma(t)\right)^{\frac{1}{1-\theta}}.
\end{equation}
Then one finds that
\[
\left\{
\begin{aligned}
	{}&\sqrt{\varphi(t)}\ln \frac{1}{Y_1(t)}=\frac{1}{2}\Sigma(t),&&Y_1(0) = 1,\\
	{}&\frac{R}{\theta-1}\left(Y_2^{1-\theta}(t)-1\right)
	=\frac{1}{2}\Sigma(t),&&Y_2(0) = 1.
\end{aligned}
\right.
\]
This also implies 
\begin{equation}\label{eq:est-Y}
	\sqrt{\varphi(t)}\ln \frac{1}{Y(t)}+\frac{R}{\theta-1}\left(Y^{1-\theta}(t)-1\right)
	\leq\Sigma(t),
\end{equation}
where $Y(t) := Y_1(t)+Y_2(t)$.
For fixed $t>0$, the function
\[
v\to \sqrt{\varphi(t)}\ln \frac{1}{v}+\frac{R}{\theta-1}\left(v^{1-\theta}-1\right)
\]
is monotonously  decreasing in terms of $v\in(0,\infty)$. Collecting \cref{eq:est-y-case1-mid,eq:est-Y} yields that
\[
y(t)\leq Y(t)=\exp\left(-\frac{\Sigma(t)}{2\sqrt{\varphi(t)}}\right)+
\left(1+\frac{\theta-1}{2R}\Sigma(t)\right)^{\frac{1}{1-\theta}}.
\]
This completes the proof of \cref{lem:est-y} with $\eta=\theta-1$. 
\subsection{The case  $\eta<\theta-1$}
The proof is in line with the previous case. 
We have
\[
\left(\frac{\sqrt{\varphi}y^{\eta+1-\theta}}{\eta+1-\theta}\right)' = 
\frac{\sqrt{\varphi}y'}{y^{\theta-\eta}} + \frac{y^{\eta+1-\theta}}{\eta+1-\theta}\cdot\frac{\varphi'}{2\sqrt{\varphi}}
\leq \frac{\sqrt{\varphi}y'}{y^{\theta-\eta}} ,
\]
which together with \cref{eq:Ly'} gives
\[
\left(\frac{\sqrt{\varphi(t)}y^{\eta+1-\theta}(t)}{\eta+1-\theta}+\frac{Ry^{1-\theta}(t)}{1-\theta}\right)'\leq -\sigma(t)\quad\Longrightarrow\quad G(\varphi(t),y(t))\geq \Sigma(t),
\]
where $G:[0,\infty)\times (0,\infty)\to \R$ is define by
\[
G(w,v): = \frac{\sqrt{w}}{\theta-\eta-1}(v^{\eta+1-\theta}-1)+\frac{R}{\theta-1}(v^{1-\theta}-1),
\]
for all $w\geq0$ and $v>0$. In addition to $Y_2(t)$ defined in \cref{eq:Y1-Y2}, we introduce
\[
Y_3(t) :=  \left(1+\frac{\theta-\eta-1}{2\sqrt{\varphi(t)}}\Sigma(t)\right)^{\frac{1}{\eta+1-\theta}}
.
\]
Since $G(w,\cdot)$ is monotonously decreasing and 
\[
G(\varphi(t),Y_2(t)+Y_3(t))\leq \Sigma(t)\leq G(\varphi(t),y(t)),
\]
we obtain 
\[
y(t)\leq\left(1+\frac{\theta-1}{2R}\Sigma(t)\right)^{\frac{1}{1-\theta}}+ \left(1+\frac{\theta-\eta-1}{2\sqrt{\varphi(t)}}\Sigma(t)\right)^{\frac{1}{\eta+1-\theta}} .
\]
This concludes the proof of \cref{lem:est-y} with $\eta<\theta-1$.

	\bibliographystyle{abbrv}

\begin{thebibliography}{10}
	
	\bibitem{Bauschke2011}
	H.~Bauschke and P.~Combettes.
	\newblock {\em Convex Analysis and Monotone Operator Theory in Hilbert Spaces}.
	\newblock CMS Books in Mathematics. Springer Science+Business Media, New York,
	2011.
	
	\bibitem{beck_fast_2009}
	A.~Beck and M.~Teboulle.
	\newblock A fast iterative shrinkage-thresholding algorithm for linear inverse
	problems.
	\newblock {\em SIAM J. Imaging Sci.}, 2(1):183--202, 2009.
	
	\bibitem{boyd_distributed_2010}
	S.~Boyd, N.~Parikh, E.~Chu, B.~Peleato, and J.~Eckstein.
	\newblock Distributed optimization and statistical learning via the alternating
	direction method of multipliers.
	\newblock {\em Foundations and Trends® in Machine Learning}, 3(1):1--122,
	2010.
	
	\bibitem{cai_linearized_2009}
	J.-F. Cai, S.~Osher, and Z.~Shen.
	\newblock Linearized {Bregman} iterations for compressed sensing.
	\newblock {\em Math. Comput.}, 78(267):1515--1536, 2009.
	
	\bibitem{candes_comp_2006}
	E.~Cand\`{e}s, J.~Romberg, and T.~Tao.
	\newblock Robust uncertainty principles : Exact signal reconstruction from
	highly incomplete frequency information.
	\newblock {\em IEEE Trans. on Information Theory}, 52(2):489--509, 2006.
	
	\bibitem{chambolle_first-order_2011}
	A.~Chambolle and T.~Pock.
	\newblock A first-order primal-dual algorithm for convex problems with
	applications to imaging.
	\newblock {\em J.Math. Imaging Vis.}, 40(1):120--145, 2011.
	
	\bibitem{chambolle_introduction_2016}
	A.~Chambolle and T.~Pock.
	\newblock An introduction to continuous optimization for imaging.
	\newblock {\em Acta Numer.}, 25:161--319, 2016.
	
	\bibitem{Chen-conv-1993}
	G.~Chen and M.~Teboulle.
	\newblock Convergence analysis of a proximal-like minimization algorithm using
	{B}regman functions.
	\newblock {\em SIAM J. Optim.}, 3(3):538--543, 1993.
	
	\bibitem{chen_first_2019}
	L.~Chen and H.~Luo.
	\newblock First order optimization methods based on {Hessian}-driven {Nesterov}
	accelerated gradient flow.
	\newblock {\em arXiv:1912.09276}, 2019.
	
	\bibitem{chen_unified_2021}
	L.~Chen and H.~Luo.
	\newblock A unified convergence analysis of first order convex optimization
	methods via strong {L}yapunov functions.
	\newblock {\em arXiv: 2108.00132}, 2021.
	
	\bibitem{chen_optimal_2014}
	Y.~Chen, G.~Lan, and Y.~Ouyang.
	\newblock Optimal primal-dual methods for a class of saddle point problems.
	\newblock {\em SIAM J. Optim.}, 24(4):1779--1814, 2014.
	
	\bibitem{davis_convergence_2016}
	D.~Davis and W.~Yin.
	\newblock Convergence rate analysis of several splitting schemes.
	\newblock {\em Splitting Methods in Communication, Imaging, Science, and
		Engineering}, pages 115--163, 2016.
	
	\bibitem{devolder_first-order_2014}
	O.~Devolder, F.~Glineur, and Y.~Nesterov.
	\newblock First-order methods of smooth convex optimization with inexact
	oracle.
	\newblock {\em Math. Program.}, 146(1-2):37--75, 2014.
	
	\bibitem{dvurechensky_computational_2018}
	P.~Dvurechensky, A.~Gasnikov, and A.~Kroshnin.
	\newblock Computational optimal transport: complexity by accelerated gradient
	descent is better than by {Sinkhorn}'s algorithm.
	\newblock In {\em Proceedings of the 35 th {International} {Conference} on
		{Machine} {Learning}}, volume~80, Stockholm, Sweden, 2018. PMLR.
	
	\bibitem{Dvurechensky2021}
	P.~Dvurechensky, M.~Staudigl, and S.~Shtern.
	\newblock First-order methods for convex optimization.
	\newblock {\em arXiv:2101.00935}, 2021.
	
	\bibitem{eckstein_splitting_1989}
	J.~Eckstein.
	\newblock {\em Splitting {Methods} for {Monotone} {Operators} with
		{Applications} to {Parallel} {Optimization}}.
	\newblock {PhD} {Thesis}, Massachusetts Institute of Technology, 1989.
	
	\bibitem{eckstein_douglas-rachford_1992}
	J.~Eckstein and D.~P. Bertsekas.
	\newblock On the {Douglas}--{Rachford} splitting method and the proximal point
	algorithm for maximal monotone operators.
	\newblock {\em Math. Program.}, 55(1):293--318, 1992.
	
	\bibitem{esser_general_2010}
	E.~Esser, X.~Zhang, and T.~F. Chan.
	\newblock A general framework for a class of first order primal-dual algorithms
	for convex optimization in imaging science.
	\newblock {\em SIAM J. Imaging Sci.}, 3(4):1015--1046, 2010.
	
	\bibitem{feijer_stability_2010}
	D.~Feijer and F.~Paganini.
	\newblock Stability of primal-dual gradient dynamics and applications to
	network optimization.
	\newblock {\em Automatica}, 46(12):1974--1981, 2010.
	
	\bibitem{goldfarb_fast_2013}
	D.~Goldfarb, S.~Ma, and K.~Scheinberg.
	\newblock Fast alternating linearization methods for minimizing the sum of two
	convex functions.
	\newblock {\em Math. Program.}, 141(1-2):349--382, 2013.
	
	\bibitem{goldstein_fast_2014}
	T.~Goldstein, B.~O'Donoghue, S.~Setzer, and R.~Baraniuk.
	\newblock Fast alternating direction optimization methods.
	\newblock {\em SIAM J. Imaging Sci.}, 7(3):1588--1623, 2014.
	
	\bibitem{guminov_universal_2019}
	S.~Guminov, A.~Gasnikov, A.~Anikin, and A.~Gornov.
	\newblock A universal modification of the linear coupling method.
	\newblock {\em Optimization Methods and Software}, 34(3):560--577, 2019.
	
	\bibitem{guminov_primal-dual_2018}
	S.~V. Guminov, Y.~E. Nesterov, P.~E. Dvurechensky, and A.~V. Gasnikov.
	\newblock Primal-dual accelerated gradient descent with line search for convex
	and nonconvex optimization problems.
	\newblock {\em arXiv:1809.05895}, 2018.
	
	\bibitem{he_convergence_2014}
	B.~He, Y.~You, and X.~Yuan.
	\newblock On the convergence of primal-dual hybrid gradient algorithm.
	\newblock {\em SIAM J. Imaging Sci.}, 7(4):2526--2537, 2014.
	
	\bibitem{he_acceleration_2010}
	B.~He and X.~Yuan.
	\newblock On the acceleration of augmented {Lagrangian} method for linearly
	constrained optimization.
	\newblock https://optimization-online.org/2010/10/2760/, 2010.
	
	\bibitem{HE2022110547}
	X.~He, R.~Hu, and Y.-P. Fang.
	\newblock Fast primal–dual algorithm via dynamical system for a linearly
	constrained convex optimization problem.
	\newblock {\em Automatica}, 146:110547, 2022.
	
	\bibitem{he_inertial_2022}
	X.~He, R.~Hu, and Y.-P. Fang.
	\newblock Inertial accelerated primal-dual methods for linear equality
	constrained convex optimization problems.
	\newblock {\em Numer. Algor.}, 90(4):1669--1690, 2022.
	
	\bibitem{huang_accelerated_2013}
	B.~Huang, S.~Ma, and D.~Goldfarb.
	\newblock Accelerated linearized {B}regman method.
	\newblock {\em J. Sci. Comput.}, 54:428--453, 2013.
	
	\bibitem{jiang_approximate_2021}
	F.~Jiang, X.~Cai, Z.~Wu, and D.~Han.
	\newblock Approximate first-order primal-dual algorithms for saddle point
	problems.
	\newblock {\em Math. Comp.}, 90(329):1227--1262, 2021.
	
	\bibitem{kadkhodaie_accelerated_2015}
	M.~Kadkhodaie, K.~Christakopoulou, M.~Sanjabi, and A.~Banerjee.
	\newblock Accelerated alternating direction method of multipliers.
	\newblock In {\em Proceedings of the 21th {ACM} {SIGKDD} {International}
		{Conference} on {Knowledge} {Discovery} and {Data} {Mining}}, pages 497--506,
	Sydney NSW Australia, 2015. ACM.
	
	\bibitem{kamzolov_universal_2019}
	D.~Kamzolov, P.~Dvurechensky, and A.~Gasnikov.
	\newblock Universal intermediate gradient method for convex problems with
	inexact oracle.
	\newblock {\em arXiv:1712.06036}, 2019.
	
	\bibitem{kang_inexact_2015}
	M.~Kang, M.~Kang, and M.~Jung.
	\newblock Inexact accelerated augmented {L}agrangian methods.
	\newblock {\em Comput. Optim. Appl.}, 62(2):373--404, 2015.
	
	\bibitem{krichene_accelerated_2015}
	W.~Krichene, A.~Bayen, and P.~Bartlett.
	\newblock Accelerated mirror descent in continuous and discrete time.
	\newblock {\em Advances in Neural Information Processing Systems (NIPS) 28},
	pages 2845--2853, 2015.
	
	\bibitem{Lan2013}
	G.~Lan and R.~Monteiro.
	\newblock Iteration-complexity of first-order penalty methods for convex
	programming.
	\newblock {\em Math. Program.}, 138(1-2):115--139, 2013.
	
	\bibitem{li_convergence_2017}
	H.~Li, C.~Fang, and Z.~Lin.
	\newblock Convergence rates analysis of the quadratic penalty method and its
	applications to decentralized distributed optimization.
	\newblock {\em arXiv:1711.10802}, 2017.
	
	\bibitem{Li2019}
	H.~Li and Z.~Lin.
	\newblock Accelerated alternating direction method of multipliers: {A}n optimal
	${O}(1/{K})$ nonergodic analysis.
	\newblock {\em J. Sci. Comput.}, 79(2):671--699, 2019.
	
	\bibitem{lin_efficient_2019}
	T.~Lin, N.~Ho, and M.~I. Jordan.
	\newblock On efficient optimal transport: {An} analysis of greedy and
	accelerated mirror descent algorithms.
	\newblock In {\em International {Conference} on {Machine} {Learning}}, pages
	3982--3991. PMLR, 2019.
	
	\bibitem{luo_accelerated_2021}
	H.~Luo.
	\newblock Accelerated differential inclusion for convex optimization.
	\newblock {\em Optimization}, https://doi.org/10.1080/02331934.2021.2002327,
	2021.
	
	\bibitem{luo_acc_primal-dual_2021}
	H.~Luo.
	\newblock Accelerated primal-dual methods for linearly constrained convex
	optimization problems.
	\newblock {\em arXiv:2109.12604}, 2021.
	
	\bibitem{luo_unified_2021}
	H.~Luo.
	\newblock A unified differential equation solver approach for separable convex
	optimization: splitting, acceleration and nonergodic rate.
	\newblock {\em arXiv:2109.13467}, 2021.
	
	\bibitem{luo_primal-dual_2022}
	H.~Luo.
	\newblock A primal-dual flow for affine constrained convex optimization.
	\newblock {\em ESAIM: Control, Optimisation and Calculus of Variations}, 28:33,
	2022.
	
	\bibitem{luo_differential_2019}
	H.~Luo and L.~Chen.
	\newblock From differential equation solvers to accelerated first-order methods
	for convex optimization.
	\newblock {\em Math. Program.}, https://doi.org/10.1007/s10107-021-01713-3,
	2021.
	
	\bibitem{bauschke_2019}
	W.~Moursi and Y.~Zinchenko.
	\newblock A {Note} on the {Equivalence} of {Operator} {Splitting} {Methods}.
	\newblock In H.~Bauschke, R.~Burachik, and D.~Luke, editors, {\em Splitting
		{Algorithms}, {Modern} {Operator} {Theory}, and {Applications}}, pages
	331--349. Springer International Publishing, Cham, 2019.
	
	\bibitem{nbmirovskii_optimal_1985}
	A.~S. Nbmirovskii and Y.~E. Nrsterov.
	\newblock Optimal methods of smooth convex minimization.
	\newblock {\em USSR Computational Mathematics and Mathematical Physics},
	25(2):21--30, 1985.
	
	\bibitem{nemirovsky_problem_1983}
	A.~Nemirovsky and D.~Yudin.
	\newblock {\em Problem {Complexity} and {Method} {Efficiency} in
		{Optimization}}.
	\newblock John Wiley \& Sons, New York, 1983.
	
	\bibitem{nesterov_introductory_2004}
	Y.~Nesterov.
	\newblock {\em Introductory {Lectures} on {Convex} {Optimization}}, volume~87
	of {\em Applied {Optimization}}.
	\newblock Springer US, Boston, MA, 2004.
	
	\bibitem{nesterov_smooth_2005}
	Y.~Nesterov.
	\newblock Smooth minimization of non-smooth functions.
	\newblock {\em Math. Program.}, 103(1):127--152, 2005.
	
	\bibitem{nesterov_gradient_2013}
	Y.~Nesterov.
	\newblock Gradient methods for minimizing composite functions.
	\newblock {\em Math. Program. Series B}, 140(1):125--161, 2013.
	
	\bibitem{nesterov_universal_2015}
	Y.~Nesterov.
	\newblock Universal gradient methods for convex optimization problems.
	\newblock {\em Math. Program.}, 152:381--404, 2015.
	
	\bibitem{nesterov_lectures_2018}
	Y.~Nesterov.
	\newblock {\em Lectures on {Convex} {Optimization}}, volume 137 of {\em
		Springer {Optimization} and {Its} {Applications}}.
	\newblock Springer International Publishing, Cham, 2018.
	
	\bibitem{ouyang_accelerated_2015}
	Y.~Ouyang, Y.~Chen, G.~Lan, and E.~Pasiliao.
	\newblock An accelerated linearized alternating direction method of
	multipliers.
	\newblock {\em SIAM J. Imaging Sci.}, 8(1):644--681, 2015.
	
	\bibitem{ouyang_lower_2021}
	Y.~Ouyang and Y.~Xu.
	\newblock Lower complexity bounds of first-order methods for convex-concave
	bilinear saddle-point problems.
	\newblock {\em Math. Program.}, 185(1-2):1--35, 2021.
	
	\bibitem{roulet_sharpness_2017}
	V.~Roulet and A.~d'Aspremont.
	\newblock Sharpness, restart, and acceleration.
	\newblock In {\em 31st {Conference} on {Neural} {Information} {Processing}
		{Systems}}, Long Beach, CA, USA, 2017.
	
	\bibitem{sabach_faster_2022}
	S.~Sabach and M.~Teboulle.
	\newblock Faster {Lagrangian}-based methods in convex optimization.
	\newblock {\em SIAM J. Optim.}, 32(1):204--227, 2022.
	
	\bibitem{stonyakin_gradient_2019}
	F.~Stonyakin, D.~Dvinskikh, P.~Dvurechensky, A.~Kroshnin, O.~Kuznetsova,
	A.~Agafonov, A.~Gasnikov, A.~Tyurin, C.~A. Uribe, D.~Pasechnyuk, and
	S.~Artamonov.
	\newblock Gradient methods for problems with inexact model of the objective.
	\newblock {\em arXiv:1902.09001}, 2019.
	
	\bibitem{stonyakin_generalized_2022}
	F.~Stonyakin, A.~Gasnikov, P.~Dvurechensky, M.~Alkousa, and A.~Titov.
	\newblock Generalized mirror prox for monotone variational inequalities:
	{Universality} and inexact oracle.
	\newblock {\em arXiv:1806.05140}, 2022.
	
	\bibitem{stonyakin_inexact_2020}
	F.~Stonyakin, A.~Gasnikov, A.~Tyurin, D.~Pasechnyuk, A.~Agafonov,
	P.~Dvurechensky, D.~Dvinskikh, A.~Kroshnin, and V.~Piskunova.
	\newblock Inexact model: {A} framework for optimization and variational
	inequalities.
	\newblock {\em arXiv:1902.00990}, 2020.
	
	\bibitem{tao_accelerated_2016}
	M.~Tao and X.~Yuan.
	\newblock Accelerated {Uzawa} methods for convex optimization.
	\newblock {\em Math. Comp.}, 86(306):1821--1845, 2016.
	
	\bibitem{tian_alternating_2018}
	W.~Tian and X.~Yuan.
	\newblock An alternating direction method of multipliers with a worst-case
	${O}(1/n^2)$ convergence rate.
	\newblock {\em Math. Comp.}, 88(318):1685--1713, 2018.
	
	\bibitem{tran-dinh_proximal_2019}
	Q.~Tran-Dinh.
	\newblock Proximal alternating penalty algorithms for nonsmooth constrained
	convex optimization.
	\newblock {\em Comput. Optim. Appl.}, 72(1):1--43, 2019.
	
	\bibitem{tran-dinh_unified_2021}
	Q.~Tran-Dinh.
	\newblock A unified convergence rate analysis of the accelerated smoothed gap
	reduction algorithm.
	\newblock {\em Optimization Letters},
	https://doi.org/10.1007/s11590-021-01775-4, 2021.
	
	\bibitem{tran-dinh_constrained_2014}
	Q.~Tran-Dinh and V.~Cevher.
	\newblock Constrained convex minimization via model-based excessive gap.
	\newblock In {\em In {Proc}. the {Neural} {Information} {Processing} {Systems}
		({NIPS})}, volume~27, pages 721--729, Montreal, Canada, 2014.
	
	\bibitem{tran-dinh_primal-dual_2015}
	Q.~Tran-Dinh and V.~Cevher.
	\newblock A primal-dual algorithmic framework for constrained convex
	minimization.
	\newblock {\em arXiv:1406.5403}, 2015.
	
	\bibitem{giselsson_smoothing_2018}
	Q.~Tran-Dinh and V.~Cevher.
	\newblock Smoothing {Alternating} {Direction} {Methods} for {Fully} {Nonsmooth}
	{Constrained} {Convex} {Optimization}.
	\newblock In P.~Giselsson and A.~Rantzer, editors, {\em Large-{Scale} and
		{Distributed} {Optimization}}, volume 2227, pages 57--95. Springer
	International Publishing, Cham, 2018.
	
	\bibitem{tran-dinh_smooth_2018}
	Q.~Tran-Dinh, O.~Fercoq, and V.~Cevher.
	\newblock A smooth primal-dual optimization framework for nonsmooth composite
	convex minimization.
	\newblock {\em SIAM J. Optim.}, 28(1):96--134, 2018.
	
	\bibitem{tran-dinh_augmented_2018}
	Q.~Tran-Dinh and Y.~Zhu.
	\newblock Augmented {Lagrangian}-based decomposition methods with non-ergodic
	optimal rates.
	\newblock {\em arXiv:1806.05280}, 2018.
	
	\bibitem{tran-dinh_non-stationary_2020}
	Q.~Tran-Dinh and Y.~Zhu.
	\newblock Non-stationary first-order primal-dual algorithms with faster
	convergence rates.
	\newblock {\em SIAM J. Optim.}, 30(4):2866--2896, 2020.
	
	\bibitem{valkonen_inertial_2020}
	T.~Valkonen.
	\newblock Inertial, corrected, primal-dual proximal splitting.
	\newblock {\em SIAM J. Optim.}, 30(2):1391--1420, 2020.
	
	\bibitem{wibisono_variational_2016}
	A.~Wibisono, A.~C. Wilson, and M.~I. Jordan.
	\newblock A variational perspective on accelerated methods in optimization.
	\newblock {\em Proc. Nati. Acad. Sci.}, 113(47):E7351--E7358, 2016.
	
	\bibitem{Wilson_2018}
	A.~Wilson, B.~Recht, and M.~Jordan.
	\newblock A {L}yapunov analysis of momentum methods in optimization.
	\newblock {\em arXiv: 1611.02635}, 2016.
	
	\bibitem{xu_accelerated_2018}
	P.~Xu, T.~Wang, and Q.~Gu.
	\newblock Accelerated stochastic mirror descent: {F}rom continuous-time
	dynamics to discrete-time algorithms.
	\newblock In {\em Proceedings of the {Twenty}-{First} {International}
		{Conference} on {Artificial} {Intelligence} and {Statistics}}, volume~84 of
	{\em Proceedings of {Machine} {Learning} {Research}}, pages 1087--1096. PMLR,
	2018.
	
	\bibitem{Xu2017}
	Y.~Xu.
	\newblock Accelerated first-order primal-dual proximal methods for linearly
	constrained composite convex programming.
	\newblock {\em SIAM J. Optim.}, 27(3):1459--1484, 2017.
	
	\bibitem{xu_iteration_2021}
	Y.~Xu.
	\newblock Iteration complexity of inexact augmented {Lagrangian} methods for
	constrained convex programming.
	\newblock {\em Math. Program.}, 185(1-2):199--244, 2021.
	
	\bibitem{yin_bregman_2008}
	W.~Yin, S.~Osher, D.~Goldfarb, and J.~Darbon.
	\newblock Bregman iterative algorithms for $\ell_1$-minimization with
	applications to compressed sensing.
	\newblock {\em SIAM J. Imaging Sci.}, 1(1):143--168, 2008.
	
	\bibitem{yurtsever_universal_2015}
	A.~Yurtsever, Q.~Tran-Dinh, and V.~Cevher.
	\newblock A universal primal-dual convex optimization framework.
	\newblock {\em arXiv: 1502.03123}, 2015.
	
	\bibitem{zhao_accelerated_2022}
	Y.~Zhao, X.~Liao, X.~He, and C.~Li.
	\newblock Accelerated primal-dual mirror dynamical approaches for constrained
	convex optimization.
	\newblock {\em arXiv:2205.15983}, 2022.
	
\end{thebibliography}

\end{document}